\documentclass[11pt]{amsart}
\usepackage{geometry}                % See geometry.pdf to learn the layout options. There are lots.
\usepackage{graphicx}
\usepackage{amssymb}
\usepackage{epstopdf}
\usepackage{lipsum}
\usepackage{rotating}
\newtheorem{theorem}{Theorem}
\newtheorem{prop}[theorem]{Proposition}
\newtheorem{cor}[theorem]{Corollary}

\theoremstyle{definition}
\newtheorem{definition}[theorem]{Definition}
\newtheorem{example}[theorem]{Example}

\def\la{{\lambda}}
\def\al{{\alpha}}
\def\be{{\beta}}
\def\ga{{\gamma}}

\def\QQ{{\mathbb Q}}

\def\CC{{\mathbb C}}

\def\st{{\tilde s}}

\def\hht{{\tilde h}}

\def\mt{{\tilde m}}

\def\mvdash{{\,\vdash\!\!\vdash}}

\def\gb{{\overline g}}

\def\dcl{{\{\!\!\{}}
\def\dcr{{\}\!\!\}}}
\def\o{\overline}

\def\MCT{{\mathcal {MCT}}}
\def\SS{\mathbb{S}}

\newcommand{\pchoose}[2]{\begin{pmatrix}#1\\ #2\end{pmatrix}}

\newdimen\squaresize \squaresize=10pt
\newdimen\thickness \thickness=0.4pt

\def\square#1{\hbox{\vrule width \thickness
     \vbox to \squaresize{\hrule height \thickness\vss
        \hbox to \squaresize{\hss#1\hss}
     \vss\hrule height\thickness}
\unskip\vrule width \thickness}
\kern-\thickness}

\def\vsquare#1{\vbox{\square{$#1$}}\kern-\thickness}
\def\blk{\omit\hskip\squaresize}

\def\young#1{
\vbox{\smallskip\offinterlineskip
\halign{&\vsquare{##}\cr #1}}}

\def\thisbox#1{\kern-.09ex\fbox{#1}}
\def\downbox#1{\lower1.200em\hbox{#1}}
\newdimen\Squaresize \Squaresize=15pt
\newdimen\Thickness \Thickness=0.4pt

\def\Square#1{\hbox{\vrule width \Thickness
     \vbox to \Squaresize{\hrule height \Thickness\vss
        \hbox to \Squaresize{\hss#1\hss}
     \vss\hrule height\Thickness}
\unskip\vrule width \Thickness}
\kern-\Thickness}

\def\Vsquare#1{\vbox{\Square{$#1$}}\kern-\Thickness}

\title[Products of symmetric group characters]{Products of characters of the symmetric group}
\author[Rosa Orellana]{Rosa Orellana}%
\address{Dartmouth College, Mathematics Department, Hanover, NH 03755, USA} \email{rosa.c.orellana@dartmouth.edu}%

\author[Mike Zabrocki]{Mike Zabrocki}%
\address{Department of Mathematics and Statistics, York University, Toronto, Ontario M3J 1P3,
Canada} \email{zabrocki@mathstat.yorku.ca}%

\date{}  % Activate to display a given date or no date

\thanks{Work supported by NSF grants DMS-1700058 and DMS-1300512, and by NSERC}

\begin{document}
\maketitle

\begin{abstract}
In \cite{OZ16}, the authors introduced a new basis of the ring of symmetric functions which evaluate to the irreducible characters of the symmetric group at roots of unity.  The structure coefficients for this new basis are the stable Kronecker coefficients. In this paper we give combinatorial descriptions for several products that have as consequences several versions of the Pieri rule for this new basis of symmetric functions.  In addition, we give several applications of the products studied in this paper. 
\end{abstract}

%\section*{Introduction}

\begin{section}{Introduction}
Schur functions, $s_\la$, form a fundamental basis for
the ring of symmetric functions.
One reason for this is that they specialize to the characters of polynomial representations of the general linear group, $GL_n$. The structure coefficients for the Schur functions are the Littlewood-Richardson coefficients, $c_{\la\mu}^\nu$,
\[s_\la s_\mu = \sum_{\nu\vdash |\la|+|\mu|} c_{\la\mu}^\nu s_\nu.\]
In \cite{OZ16}, we introduced a new basis of symmetric functions, $\st_\la$, that we believe that in many ways is as fundamental as the basis of Schur functions since these functions specialize to the irreducible characters of the symmetric group, $S_n$.
The structure coefficients for the $\st$-basis,
\[\st_\la\st_\mu = \sum_{|\nu|\leq |\la|+|\mu|} \bar{g}_{\la\mu}^\nu \st_\nu~,\]
are the stable Kronecker coefficients, $\bar{g}_{\la\mu}^\nu$, \cite{BOR, Lit1,Murg3, Val} and they can be seen to be a generalization of the
Littlewood-Richardson coefficients in the sense that $c_{\la\mu}^\nu = \bar{g}_{\la\mu}^\nu$ when $|\nu|=|\la|+|\mu|$.
The $\st$-basis gives an elegant and natural symmetric function formulation for the stable Kronecker coefficients.

In addition, the irreducible character basis also reformulates the restriction problem \cite{ButlerKing, Lit2, King, Nis, ScharfThibon, STW} which asks for a combinatorial interpretation for the multiplicities when we restrict a $GL_n$ polynomial representation to $S_n$. Using the $\st$-basis, we are simply asking for the change of basis coefficients between the Schur basis and the $\st$-basis.

A notorious open problem is to find a positive combinatorial 
interpretation for the Kronecker coefficients, $g_{\la,\mu}^\nu$, 
these are the coefficients in the decomposition of the tensor product 
of two irreducible representations of $S_n$. The stable Kronecker 
coefficients are the Kronecker coefficients in the sense that for any 
three partitions $\alpha,\beta$ and $\gamma$, $\overline{g}_{\alpha,
\beta}^\gamma = g_{(n-|\alpha|,\alpha),(n-|\beta|,\beta)}^{(n-|
\gamma|,\gamma)}$ for $n\gg 0$. Many attempts have been made to find 
combinatorial interpretations and there are a large number of papers in the literature addressing several aspects of the Kronecker coefficients, for example \cite{BO, Blasiak, BOR, Man, PPV, Rem, RW, Rosas, ScharfThibonWybourne, Val}. Recent progress, however, has lead to combinatorial interpretations for very few special cases. In terms of the $\st$-basis, the only combinatorial interpretations known are for $\st_{1^k}\st_\lambda$ and $\st_k\st_\mu$ \cite{Blasiak, BO, Liu, BlasiakLiu}. These interpretations involve different objects for which there are no obvious generalizations.

%which in special cases, are equal to the $\st$-decomposition of the products
%$\st_{1^k}\st_\lambda$ and $\st_k\st_\mu$
%\cite{Blasiak, BO, Liu, BlasiakLiu}.

The $\st$-expansion of a symmetric function specializing to an $S_n$-character  corresponds to the decomposition of the corresponding representation into irreducible $S_n$ characters. In this setting, products of symmetric functions correspond to tensor products of representations. 
%when viewed as $S_n$-characters correspond to characters of tensor products of the corresponding representations.
In \cite{OZ16} we also introduced the $\hht_\mu$ symmetric functions which specialize to values of the trivial characters of a Young subgroup 
$S_{\mu_1} \times S_{\mu_2} \times \cdots \times S_{\mu_{\ell(\mu)}} \subseteq S_n$
induced to the symmetric group $S_n$.  The focus of this work are products involving the $S_n$-characters $h_\mu$, $\hht_\mu$ and $\st_\mu$, where $h_\mu$ are the complete homogeneous functions. 
%The symmetric function $h_k$ specializes to the character of the $k^{th}$ symmetric power of a ($Gl_n$ or $S_n$) $n$-dimensional module and $h_\mu = h_{\mu_1} h_{\mu_2} \cdots h_{\mu_\ell}$ specializes to the character of a tensor product of these modules.
%The symmetric function $\hht_\mu$ specializes to the character of the induced trivial representation from a Young subgroup to $S_n$ and $\st_\mu$ specializes to the irreducible character indexed by $(n - |\mu|, \mu)$.

The aim of this paper is to give combinatorial interpretations for the coefficients in the $\st$-expansion for the products  $\st_{\mu_1}\st_{\mu_2}\cdots \st_{\mu_\ell}\st_\lambda$, $\hht_{\mu_1}\hht_{\mu_2}\cdots \hht_{\mu_\ell}\st_\lambda$, $\hht_\mu \st_\la$, and $h_\mu \st_\la$.  Our combinatorial description for $h_\mu \st_\la$ is in terms of column strict tableaux filled with multisets containing at most one bar entry and unbarred entries
that satisfy a lattice condition by reading the barred entries
with respect to the unbarred entries.
When we look at the products  $\hht_{\mu_1}\hht_{\mu_2}\cdots \hht_{\mu_\ell}\st_\lambda$ we have the same objects but now we cannot repeat unbarred entries, that is cells are filled with sets. And for the products  $\st_{\mu_1}\st_{\mu_2}\cdots \st_{\mu_\ell}\st_\lambda$ we further require that only sets with more than one element can go on the first row of the tableau. Finally, for the products $\hht_\mu \st_\la$ we have the conditions for $ h_\mu\st_\la$ and the additional condition that at most one barred and at most one unbarred entry can be in each cell. This provides a unified description using the same combinatorial objects for all these products.

%The aim of this paper is to give the irreducible character basis expansion, the $\st$-expansion, for the products $\st_{\mu_1}\st_{\mu_2}\cdots \st_{\mu_\ell}\st_\lambda$, $\hht_{\mu_1}\hht_{\mu_2}\cdots \hht_{\mu_\ell}\st_\lambda$, $\hht_\mu \st_\la$, and $h_\mu \st_\la$ in terms of multiset filled tableaux. 

For symmetric functions $f$ and $g$, we say that $f\leq g$ if the coefficients of
$g-f$ when expanded in the $\st$-basis are non-negative. The combinatorial interpretations for the products described in this paper, clearly illustrate the following inequalities.

$$\st_{\mu_1} \st_{\mu_2} \cdots \st_{\mu_\ell} \st_\la \qquad \leq\qquad
\hht_{\mu_1} \hht_{\mu_2} \cdots \hht_{\mu_\ell} \st_\la \qquad\leq\qquad
h_\mu \st_\la$$

After some preliminaries we define a scalar product
on symmetric functions for which the set of
functions $\{\st_\la\}_{\la}$
forms an orthonormal basis (Section \ref{sec:scalarproduct}).
Our main results are the combinatorial rules for the products
$\st_{\mu_1} \st_{\mu_2} \cdots \st_{\mu_\ell} \st_\la$,
$\hht_{\mu_1} \hht_{\mu_2} \cdots \hht_{\mu_\ell} \st_\la$ and
$h_\mu \st_\la$ (Theorem \ref{thrm:hmustlam}, 
\ref{thrm:htkmultiplestla} and \ref{thrm:stkmultiplestla}) that appear
in Section \ref{sec:rules} and we prove these formulas in
Section \ref{sec:proofs} after also giving a combinatorial interpretation
for the product $\hht_\mu \st_\la$ in Section \ref{sec:htstprod}.
In Section \ref{sec:connections}, we end the paper with a discussion
of how our combinatorial interpretations
give a unified way to view several different mathematical questions
and constructions considered in the literature (e.g.
the restriction problem of an irreducible $GL_n$ module to $S_n$, 
quantum entanglements of $q$-bits,
Grothendeick symmetric functions and dimensions of irreducible representations of the partition and quasi-partition algebras).

\end{section}

%\todo{
%1. write an outline of introduction and then the actual introduction

%2. delete the examples (deleted one or two and commented out a number of others)

%3. make sure that the proofs are correct

%4. trim down the relevant references
%we won't need all of the references below, but we should try
%our hardest to connect these symmetric functions to the mathematical
%literature.  Some of these references were taken from the last paper

%5. need to clean up the LR rule section to know precisely
%which forms we are using.  Props \ref{prop:LR1}, \ref{prop:LR2}, \ref{prop:LR3}.

%6. Not such a long introduction and stay on message.

%7. Does the introduction tells a story? 
%}

\begin{section}{Preliminaries}

\begin{subsection}{Partition and diagram notation}
A partition of a non-negative integer $n$ is a list of
positive integers $\lambda = (\lambda_1, \lambda_2, \ldots, \lambda_{\ell})$
with $\lambda_i \geq \lambda_{i+1} > 0$ and $n = \sum\limits_{i=1}^{\ell}\lambda_i$.
We refer to $n$ as the size of the partition (denoted $|\lambda| := n$ or $\lambda\vdash n$).
The nonzero entries $\lambda_i$ are called the parts of the partition. The number of nonzero parts $\ell= \ell(\lambda)$ is called the length of $\lambda$.
We will often refer to the partition $\lambda$ with the first part removed as ${\overline \lambda} = (\lambda_2, \lambda_3, \ldots, \lambda_{\ell(\la)})$.
Similarly, if $a$ is an integer with $a \geq \la_1$, then $(a,\la)$ represents
the partition $(a, \la_1, \la_2, \ldots, \la_{\ell(\la)})$.
The set of all partitions is denoted by $\mbox{Par}$
and $\mbox{Par}_{\leq n}$ will denote partitions of
size less than or equal to $n$.

Partitions will be visually represented by their
Young diagrams with $n$ cells arranged in $\ell(\lambda)$ rows
with $\lambda_i$ cells in the $i^{th}$ row.  We will represent our partitions in French notation with the
largest part of the partition on the bottom and the coordinates of the cells are $(i,j)$ such that $1 \leq i \leq \ell(\la)$ and $1 \leq j \leq \la_i$.
The number of parts of the partition $\lambda$ which are of size $i$ will be denoted
$m_i(\lambda) := |\{ j : \lambda_j = i \}|$.

For any partition $\la$ we define 
%The size of the stabilizer of an element of $S_n$ with cycle structure $\lambda$ under the action of conjugation is
%equal to
$$z_\lambda = \prod_{i=1}^{\lambda_1} i^{m_i(\lambda)} m_i(\lambda)!~.$$

The conjugate of a partition $\la$ is denoted $\la'$ and it is the partition with the $i^{th}$ row equal to $\sum\limits_{j \leq i} m_j(\lambda)$.

For two partitions $\la$ and $\mu$ where $\ell(\lambda) \geq \ell(\mu)$ and $\lambda_i \geq \mu_i$ for all $1 \leq i \leq \ell(\mu)$. We define the skew partition $\la/\mu$ to be the cells of the partition $\la$ which are not in the cells of $\mu$.  We say that $\lambda/\mu$ is a horizontal strip
if $\lambda_{i} \leq \mu_{i+1}$ for $1 \leq i \leq \ell(\mu)-1$.
\end{subsection}

\begin{subsection}{Multisets and multiset partitions}
We will distinguish a multiset from a set using the notation
$\dcl 1^{a_1}, 2^{a_2}, \ldots, m^{a_m} \dcr$ to indicate that
the multiset contains $a_i$ copies of $i$ and the elements are considered without order.
%This notion of a combinatorial object comes from from monomials where the operation of union $\cup$ corresponds to product and subset $\subseteq$ corresponds to a monomial which divides another.
A multiset partition $\pi$ of a multiset $S$ is a
multiset of non-empty multisets whose union (with multiplicities) is $S$.  That is,
\begin{equation}
\pi = \dcl S_1, S_2, \ldots, S_{\ell(\pi)} \dcr
\end{equation}
with $S_1 \cup S_2 \cup \cdots \cup S_{\ell(\pi)} = S$ (here the union is taken with multiplicities).
We have denoted $\ell(\pi)$ as the number of parts of the
multiset partition.  If $S$ is a multiset then we will
indicate that $\pi$ is a multiset partition of $S$ with the
notation $\pi \mvdash S$.

\begin{subsection}{Multiset valued tableaux} \label{sec:multisettab}

We now describe properties of the tableaux  which are common to all the combinatorial interpretations presented in the following sections.

A \emph{tableau} $T$ is a map from a the set of cells of the Young diagram to a set of labels. We say that a tableau is \emph{column strict} (with respect to some order on the labels) if $T_{(i,j)} \leq T_{(i,j+1)}$ for all $1 \leq i \leq \ell(\la)$ and $1 \leq j < \la_i$ and if $T_{(i,j)} < T_{(i+1,j)}$ for all $1 \leq i < \ell(\la)$ and $1 \leq j \leq \la_{i+1}$.   We are interested in column strict tableaux with labels being multisets, thus we need to define the order on multisets that we will use throughout the paper.   We assume that $\o1< \o2<\cdots < 1<2< \cdots$. Let $M=\dcl m_1\leq m_2\leq \dots \leq m_r\dcr$ and $N=\dcl n_1\leq n_2 \leq \dots \leq n_s\dcr$ be two multisets with elements from the set $\{\o1,\o2, \ldots, 1, 2, \ldots\}$ and at most one barred entry. We define $M<N$ if $m_{r-i}=n_{s-i}$ for $0\leq i\leq k$ and $m_{r-k-1}<n_{s-k-1}$ or $r-i=0$ and $s-i>0$.  This is the \emph{reverse lexicographic} order (\emph{reverse lex}) on the entries of the multisets.  For example, $\dcl \o{5}, 1,1,1,2,2,3,4\dcr < \dcl\o{2}, 1,1,2,3,3,4 \dcr$.

%In our tableaux we fill the cells with non-empty multisets chosen from a completely ordered alphabet $A=\{a_1, a_2,\dots\}$. 

The \emph{content} of a tableau $T$ is defined as the multiset which contains $a_i^{m_i}$ where $a_i\in \{\o1,\o2, \ldots, 1, 2, \ldots\}$ occurs $m_i$ times in $T$.     The cells of the tableaux will be filled with non-empty multisets with elements chosen from the set $\{\bar{1}, \bar{2},\bar{3} \ldots, 1, 2, 3,\ldots \}$ such that each multiset contains at most one barred entry. Thus, the content of the tableaux is $\dcl \o1^{\al_1}, \o2^{\al_2}, \ldots, \o{\ell}^{\al_\ell}, 1^{\be_1}, 2^{\be_2}, \ldots, k^{\be_k} \dcr$ for some weak compositions $\alpha$ and $\beta$.

The \emph{shape} of a tableau $T$ is the sequence obtained by reading the lengths of each row in $T$. We denote by $sh(T)$ to be the shape of $T$.
All of our tableaux will be of shape $(r,\gamma)/(\gamma_1)$ for a partition $\gamma$ and some integer $r\geq \ga_1$.

\begin{definition} Let $\gamma$ be a partition, $\alpha$ and $\beta$ be compositions. Then the set $\MCT_\gamma{(\alpha,\beta)}$ contains tableaux $T$ that are column strict with respect to the reverse lex order, have shape $(r,\gamma)/(\gamma_1)$, and content $\dcl \o1^{\al_1}, \o2^{\al_2}, \ldots, \o{\ell}^{\al_\ell}, 1^{\be_1}, 2^{\be_2}, \ldots, k^{\be_k} \dcr$ with at most one barred entry in each cell. Further, we require that cells in the first row cannot be filled with multisets containing only barred entries.

\end{definition}

In our examples of tableaux, to save space, the
$\dcl~\dcr$ are dropped from the labels of the entries.

\begin{example}\label{sampletab} The following tableau is in $\MCT_{(3,3,2,1)}((3,2,1),(5,2,1,1))$

\squaresize=16pt
$$T =\raisebox{-30pt}{\young{3\cr\o111&\o24\cr\o2&2 &\o12\cr\o1&1&1\cr&&&\o31\cr}}$$
Then $sh(T) =(4,3,3,2,1)/(3)$ and $\o{sh(T)} = (3,3,2,1)$. 
\end{example}

For any tableau $T$, let $read(T)$ be the word of entries in the tableau from bottom row to top row and from right to left in the rows.  If $S$ is a multiset of non-barred entries then let $read(T|_{S})$ represent the reading word
of the barred entries in the cells with $\dcl \o{j}, S \dcr$
as a label and let $read(T|_{-})$ represent the reading word of the cells which have only a barred entry, i.e., when $S=\emptyset$.  The multisets with no barred entries do not contribute to the reading word.

In the tableau of Example \ref{sampletab}, we have $read(T|_{-})=\o1\o2$, $read(T|_{1})=\o3$, $read(T|_{1^2}) = \o1$  $read(T|_{2})=\o1$, $read(T|_{3})=\emptyset$ and $read(T|_{4})=\o2$.  We remark that we have omitted the brackets in the indexing sets to unclutter the notation.

We indicate the concatenation of several words of the form 
$read(T|_S)$ by placing dots at the end of each word $read(T|_S)$. Consecutive dots indicates that some words are empty.   For example, $read(T|_{-})read(T|_{1})read(T|_{1^2})read(T|_{2})read(T|_{3})read(T|_{4})=\o1\o2.\o3.\o1.\o1..\o2$ in the tableau of Example \ref{sampletab}.

Let $w$ be a word of content $\dcl 1^{\lambda_1}, 2^{\lambda_2}, \ldots, \ell^{\lambda_\ell}\dcr$.  For any subword of
$u$ of $w$, let $n_i(u)$ be the number of occurrences of $i$ in $u$.
A word $w$ is called {\it lattice} if for all subwords 
$u$ and $v$ such that $w = uv$,
$n_i(u) \geq n_{i+1}(u)$ for all $1 \leq i \leq \ell$.
We use this definition to define a lattice multiset tableau.

\begin{definition} If $T \in \MCT_\gamma{(\lambda,\mu)}$
for some partitions $\lambda, \mu$ and $\gamma$, then
let $S_1 < S_2 < \ldots < S_d$ be the multisets
of the unbarred entries that appear in $T$
(ignoring the barred entries), then we say that $T$ is
a \emph{lattice tableau} if the word
$$read(T|_{-}) read(T|_{S_1}) read(T|_{S_2})
\cdots read(T|_{S_d})$$
is lattice.
\end{definition}

\begin{example}
Consider the two tableaux
\squaresize=14pt
$$\raisebox{-10pt}{\young{\o1&\o11&\o21&\o22\cr&&&\cr}}
\qquad \hbox{ and }\qquad
\raisebox{-10pt}{\young{\o1&\o21&\o21&\o12\cr&&&\cr}}$$
which are both elements of $\MCT_{(4)}((2,2),(2,1))$. 
The tableau on the left is lattice because 
$read(T|_-)read(T|_1)read(T|_2) = \o1.\o1\o2.\o2$; however, 
the tableau on the right is not lattice since
$read(T|_-)read(T|_1)read(T|_2) = \o1.\o2\o2.\o1$.
\end{example}

\end{subsection}

\end{subsection}

\begin{subsection}{Symmetric function notation} \label{subsec:sf}
Define the Hopf algebra of symmetric functions to be the polynomial ring 
$$\Lambda = \QQ[h_1, h_2, h_3, \ldots]
%= \QQ[e_1, e_2, e_3, \ldots] 
= \QQ[p_1, p_2, p_3, \ldots]$$
where the generators are related by the equation
$$h_m = \sum_{\lambda \vdash m} \frac{p_\lambda}{z_\lambda}. $$

%$$h_m = \sum_{i=1}^m (-1)^{i-1} h_{m-i} e_i \hskip.2in  
%\hskip.2in
%e_m = \sum_{\lambda \in \bP_m} \frac{(-1)^{m+\ell(\lambda)} p_\lambda}{z_\lambda}~.$$
%$$e_m = \sum_{i=1}^m (-1)^{i-1} e_{m-i} h_i \hskip.2in p_m = m h_m - \sum_{r=1}^{m-1} p_r h_{m-r} \hskip.2in
%p_m = (-1)^{m-1} m e_m - \sum_{r=1}^{m-1} (-1)^{r-1} p_r e_{m-r}~.$$
Then, for a partition $\lambda$, the complete and power sum symmetric functions are defined by
$$h_\lambda:= h_{\lambda_1} h_{\lambda_2} \cdots h_{\lambda_{\ell(\lambda)}}\hskip.2in
%e_\lambda:= e_{\lambda_1} e_{\lambda_2} \cdots e_{\lambda_{\ell(\lambda)}}\hskip.1in\hbox{ and }\hskip.1in
p_\lambda:= p_{\lambda_1} p_{\lambda_2} \cdots p_{\lambda_{\ell(\lambda)}}.$$
Since $\Lambda$ is a polynomial ring, it has as linear bases
$\{ h_\lambda \}_{\lambda \in \mbox{Par}}$ and $\{ p_\lambda \}_{\lambda \in \mbox{Par}}$, where $\mbox{Par}$ denotes the set of all partitions.

A definition of the Schur basis, $\{s_\la\}_{\la \in \mbox{Par}}$,
is by the formula
$$
h_\mu = \sum_{\la\vdash |\mu|} K_{\la\mu}s_\la ,
$$
where $K_{\la\mu}$ is the Kostka coefficient and it is equal to the
number of column strict tableau of shape $\la$ and content $\mu$.

The ring of symmetric function is endowed with a scalar product for which the power sum symmetric functions are orthogonal, 
$\langle p_\lambda, p_\mu \rangle = z_{\lambda} \delta_{\lambda\mu}$
where $\delta_{\lambda\lambda} =1$ and $\delta_{\lambda\mu} = 0$ if $\lambda \neq \mu$.
%With respect to this scalar product, the 
%monomial basis $\{ m_\lambda\}_{\lambda \in \bP}$ is the graded
%dual of the complete symmetric functions
%$\langle h_\lambda, m_\mu \rangle = \delta_{\lambda\mu}$
and the Schur basis is self dual $\langle s_\lambda, s_\mu \rangle = \delta_{\la\mu}$.
A useful application of the scalar product
and the dual basis is to 
give an expansion of one symmetric function in one of these bases since
for any symmetric function $f$,
\begin{equation}\label{eq:symfun}
f = %\sum_{\lambda \in \bP} \left< f, m_\lambda \right> h_\lambda
\sum_{\lambda \in \mbox{Par}} \left< f, s_\lambda \right> s_\lambda
= \sum_{\lambda \in \mbox{Par}} \left< f, p_\lambda \right> \frac{p_\lambda}{z_\lambda}~.
\end{equation}

The structure coefficients of Schur functions are known as the Littlewood-Richardson coefficients, denoted by $c_{\la,\mu}^\nu $,
\begin{equation}
s_\lambda s_\mu = \sum_{\nu \vdash {|\la|+|\mu|}} c_{\la\mu}^\nu s_\nu ~,
\end{equation}
and in general, the coefficient of $s_\lambda$ in the
product $s_{\nu^{(1)}} s_{\nu^{(2)}}
\cdots s_{\nu^{(\ell)}}$ is denoted $c_{\nu^{(1)}\nu^{(2)}\cdots\nu^{(\ell)}}^\lambda$.
Combinatorial interpretations for the coefficients $c_{\lambda\mu}^\nu$
are given in Section \ref{sec:LRrules}.

The Kronecker product is a second product on symmetric functions defined such that
$\frac{p_\la}{z_\la} \ast \frac{p_\mu}{z_\mu} = \delta_{\la\mu} \frac{p_\la}{z_\la}$.
The Kronecker coefficients, $g_{\la\mu}^\nu$, are the structure
coefficients with respect to the Schur basis 
as in the coefficients of the equation
\begin{equation}
s_\lambda \ast s_\mu = \sum_\nu g_{\la\mu}^\nu s_\nu~.
\end{equation}
We assume that the reader is familiar with basic facts about the Kronecker coefficients
including the basic stability properties \cite{BOR, Murg2, Murg3, Val} and that $\left<f \ast g, h \right> =
\left< f, g \ast h \right>$.  We will use the notation
$\gb_{\la\mu}^\nu := g_{(n-|\la|,\la)(n-|\mu|,\mu)}^{(n-|\nu|,\nu)}$ for the stable limit of these coefficients if $n$ is sufficiently large.

%We also need to reference the
%change of basis coefficients that occur in the expressions
%$$h_\mu = \sum_{\lambda \in \bP_{|\mu|}} K_{\la\mu} s_\lambda 
%\hskip.1in\hbox{ or }\hskip.1in s_\lambda = \sum_{\mu \in \bP_{|\la|}} K_{\la\mu} m_\mu~.$$
%They are known as the Kostka coefficients and
%$K_{\lambda\mu}$ is equal to the number of column
%strict tableaux of shape $\lambda$ and content $\mu$.

\end{subsection}

\begin{subsection}{Littlewood-Richardson rules} \label{sec:LRrules}
Our combinatorial interpretations are built upon
tableaux constructions of the Littlewood-Richardson
rule.  We will need the following two forms of the
Littlewood-Richardson rule that we will use in the
proofs below.

For a tableau $T$ of content
$\dcl 1^{\lambda_1}, 2^{\lambda_2}, \ldots, \ell^{\lambda_\ell}\dcr$ (potentially skew shape),
we say that $T$ is lattice if $read(T)$ is lattice
(see Section \ref{sec:multisettab} for the definition of
$read(T)$ and lattice word).
The other form of the Littlewood-Richardson rule uses
jeu de Taquin and we will
assume that the reader is familiar with the relevant
definitions to use these forms of the statement of
the Littlewood-Richardson rule.

\begin{prop}\label{prop:LR1}
(\cite{Ful}, Exe. 1, Sect. 5.2, pp. 66)
The Littlewood-Richardson coefficient
$c^\nu_{\la\mu}$
is equal to the number of pairs of tableaux $(S,T)$
where $S$ is of shape $\la$ and $T$ is of shape $\mu$
such that
$read(S)read(T)$ is a lattice word of content
$\dcl 1^{\nu_1}, 2^{\nu_2}, \ldots, \ell^{\nu_{\ell(\nu)}}\dcr$.
\end{prop}

We will represent these pairs $(S,T)$ graphically by placing $S$ below and to the right of $T$
so that $read(S)read(T)$ is the extension of the reading of the skew tableau.

\begin{example}
The coefficient $c_{(2,1),(3,1,1)}^{(4,2,1,1)} = 2$
is equal to the number of the following tableaux
$$\young{4\cr
1&2\cr
\blk&\blk&3\cr
\blk&\blk&2\cr
\blk&\blk&1&1&1\cr}\hskip 1in
\young{2\cr
1&4\cr
\blk&\blk&3\cr
\blk&\blk&2\cr
\blk&\blk&1&1&1\cr}
$$
where the reading words of these tableaux are $11123214$ and $11123412$ respectively.
If we change the roles of $(2,1)$ and $(3,1,1)$,
$$\young{4\cr
3\cr
1&1&2\cr
\blk&\blk&\blk&2\cr
\blk&\blk&\blk&1&1\cr}\hskip 1in
\young{4\cr
2\cr
1&1&3\cr
\blk&\blk&\blk&2\cr
\blk&\blk&\blk&1&1\cr}
$$
where the reading words of these tableaux are $11221134$ and $11231124$ respectively.
\end{example}

\begin{prop}\label{prop:LR2} (\cite{Ful}, Corollary 1-4,  \cite{Stanley}, Fomin's Appendix, Theorem A1.3.1)
The Littlewood-Richard coefficient $c_{\lambda\mu}^\nu$ is equal to the number of
column strict tableaux of shape $\nu/\mu$ which are jeu de Taquin equivalent to a particular
(arbitrary) column strict tableau of shape $\lambda$.
\end{prop}

\begin{example}\label{ex:LRFomin}
The coefficient $c_{(2,1),(3,1,1)}^{(4,2,1,1)} = 2$
is equal to the number of the following tableaux.
$$\young{2\cr
\cr
&1\cr
&&&1\cr}\hskip 1in \young{1\cr
\cr
&2\cr
&&&1\cr}$$
or by interchanging the roles of $(2,1)$ and $(3,1,1)$,
$$\young{3\cr
2\cr
&1\cr
&&1&1\cr}
\hskip 1in \young{3\cr
1\cr
&2\cr
&&1&1\cr}$$
\end{example}
\end{subsection}

%Another version of the Littlewood-Richardson in terms of lattice words is the following:

%\begin{prop}\label{prop:LR3} (Corollary 1-4 of \cite{Ful} and Theorem A1.3.3, \cite{Stanley})
%The Littlewood-Richard coefficient $c_{\lambda\mu}^\nu$ is equal to the number of column strict tableaux of shape $\nu/\mu$ and content $\lambda$ whose reading word is a lattice permutation. 
%\end{prop}
%
%\begin{example}
%The reading words of the
%tableaux in Example \ref{ex:LRFomin} are
%$112$, $121$, $11123$, $11213$ and these words are all lattice,
%while $12122132$ is not a lattice word since the head of the word $12122$
%has more $2$s than $1$s.
%\end{example}

%\end{section}
%\begin{section}{A characterization for the $\st$-basis}

%In this section we introduce a new inner product on symmetric functions and show that the $\st$-basis is orthonormal with respect to this inner product. 

\begin{subsection}{Characters as symmetric functions}
Every partition $\la$ with $\ell(\la)\leq n$ indexes an irreducible polynomial module of $GL_n(\CC)$ with character at an element with eigenvalues $x_1, x_2,\ldots, x_n$ equal to $s_\lambda[X_n]$ where $X_n = x_1 + x_2 + \cdots + x_n$.

Now the symmetric group, $S_n$, as permutation matrices, is contained in $GL_n(\CC)$. 
%In \cite{OZ16}, we showed that the characters of the symmetric group $S_n$ are symmetric polynomials in the eigenvalues of the permutation matrices and we can consider them as specializations of symmetric functions. 
In \cite{OZ16}, we introduced two new bases  of symmetric functions $\{\st_\la\}_{\la\in\mbox{Par}}$ and $\{\hht_\la\}_{\la\in\mbox{Par}}$ that have the property that they specialize to characters of the symmetric group when evaluated at the eigenvalues of permutation matrices. In particular, the functions $\st_\la$ specialize to irreducible $S_n$ characters. In what follows, we cite some results from \cite{OZ16} that we will need in this paper. 

The $\st_\la$ and $\hht_\la$ are defined via change of basis formulae. For a set partition $\pi$ of length $\ell(\pi)$, we define $\tilde{m}(\pi)$ to be the partition of $\ell(\pi)$ consisting of the multiplicities of the multisets occurring in $\pi$, for example,  $\tilde{m}(\dcl \dcl 1,1,2 \dcr, \dcl 1,1,2 \dcr, \dcl 1,4 \dcr \dcr)=(2,1)~.$

%We wish to define and work with the characters of the symmetric group.
%Define three symmetric functions by the following three change of basis formulae.
\begin{definition} Fix a partition $\mu$
and a positive integer $n \geq 2|\mu|$,
\begin{align}
h_\mu &= \sum_{\pi \mvdash \dcl 1^{\mu_1}, 2^{\mu_2}, \ldots, \ell(\mu)^{\ell(\mu)}\dcr} \hht_{\mt(\pi)}\label{eq:htdef}\\
\hht_\mu &= \sum_{\la \in \mbox{Par}_{\leq |\mu|}} K_{(n-|\la|,\la), (n-|\mu|,\mu)} \st_\la\label{eq:stdef}
= \sum_{\nu \vdash |\mu|} \sum_{\lambda} K_{\nu\mu} \st_\lambda
\end{align}
where the $K_{\lambda\mu}$ are the Koskta coefficients and the inner sum on the right
is over all partitions $\lambda$ such that $\nu/\lambda$ is a horizontal strip.
\end{definition}

\begin{theorem}\label{thrm:charsummary} (Equations (6)--(10) of \cite{OZ16})
For a partition $\lambda$, 
the symmetric functions $\st_\lambda$ and $\hht_\lambda$ have the property that
for a positive integer $n \geq |\lambda|+ \lambda_1$ and $\mu \vdash n$,
\begin{equation}
\st_\la[\Xi_\mu] = \left< s_{(n-|\lambda|, \lambda)}, p_\mu \right> = \chi^{(n-|\lambda|, \lambda)}(\mu)
\hskip .1in\hbox{ and }\hskip .1in
\hht_\la[\Xi_\mu] = \left< h_{(n-|\lambda|, \lambda)}, p_\mu \right>~.
\end{equation}
\end{theorem}

Note that this implies that the characters $\st_\lambda$ and $\hht_\lambda$ are the characters of representations of $S_n \subseteq GL_n(\CC)$ in the same way that the Schur functions are the characters of the irreducible polynomial representations
of $GL_n(\CC)$.  They have the reduced Kronecker coefficients as structure coefficients (see Theorem 4
of \cite{OZ16}) so that 
$$\st_\la \st_\mu = \sum_{\nu: |\nu|\leq|\la|+|\mu|} \gb_{\la\mu}^\nu \st_\nu.$$
\end{subsection}

\begin{subsection}{A scalar product on characters}
\label{sec:scalarproduct}
The usual scalar product on characters can be used to define a scalar product
on symmetric functions.  Choose an $n$ sufficiently large (any $n$ satisfying
$n \geq 2 \max(deg(f), deg(g))$ will do), then define
\begin{equation}\label{eq:scalarproduct}
\left< f, g \right>_@ = \sum_{\nu \vdash n} \frac{f[\Xi_\nu] g[\Xi_\nu]}{z_\nu}
= \frac{1}{n!}\sum_{\sigma \in S_n} f[\Xi_{cyc(\sigma)}] g[\Xi_{cyc(\sigma)}]~.
\end{equation}
where $cyc(\sigma)$ is a partition representing the cycle structure of $\sigma \in S_n$.
We use the $@$-symbol as a subscript of the
right angle bracket to differentiate this scalar product from the
usual scalar product where $\left< s_\lambda, s_\mu \right> = \delta_{\la\mu}$.
One should recognize this scalar product as a translation of the usual scalar product
on characters of a finite group when the group is $S_n$.

At first glance, there should be some concern about the dependence of the right hand side of
this expression on the value $n$, 
but quite surprisingly this expression is independent of $n$ as long as it
is sufficiently large. This is not an obvious property, but it follows from the fact that the $\st_\la$ elements form a basis of the symmetric functions, linearity of the scalar product, and the following proposition.

\begin{prop} For all partitions $\la$ and $\mu$,
\begin{equation}
\left< \st_\la, \st_\mu \right>_@ = \delta_{\lambda\mu}~.
\end{equation}
\end{prop}

\begin{proof}  Let $n$ be greater than or equal to $2 \max(|\la|,|\mu|)$, then $(n-|\la|,\la)$
and $(n-|\mu|,\mu)$ are partitions and we calculate
\begin{align}
\left< \st_\la, \st_\mu \right>_@ &= \sum_{\nu \vdash n} \frac{\st_\la[\Xi_\nu] \st_\mu[\Xi_\nu]}{z_\nu}\\
&= \sum_{\nu \vdash n} \frac{\chi^{(n-|\lambda|,\lambda)}(\nu)
\chi^{(n-|\mu|,\mu)}(\nu)}{z_\nu} = \delta_{\lambda\mu}~.
\end{align}
This last equality follows by the character relations.
\end{proof}

We can relate these scalar products by using the Frobenius
map which is a linear isomorphism from the class functions of
the symmetric group to the ring of symmetric functions.
Since we know that characters of the symmetric group (and
hence class functions) can be expressed as symmetric
functions, we can
define the Frobenius map or characteristic map on symmetric
functions
\begin{equation}\label{eq:frob}
\phi_n(f) = \sum_{\nu \vdash n} f[\Xi_\nu] \frac{p_\nu}{z_\nu}~.
\end{equation}
Here we have that
$\phi_n$ is a map from the ring of symmetric functions to the
subspace of symmetric functions of degree $n$ with the property for symmetric
functions $f$ and $g$,
$\phi_n(f g) = \phi_n(f) \ast \phi_n(g)$.
Theorem
\ref{thrm:charsummary} may be restated in terms of this map to say that
\begin{equation}
\phi_n( \st_\lambda ) = s_{(n - |\la|, \la)}
\hskip .1in\hbox{ and }\hskip .1in
\phi_n( \hht_\lambda ) = h_{(n - |\la|, \la)}~.
\end{equation}

We have the following proposition.
\begin{prop} If $n\geq 2 \max(deg(f),deg(g))$,
then
\begin{equation}\label{eq:twoscalarrelations}
\left< f, g \right>_@ = \left< \phi_n(f), \phi_n(g) \right>
\end{equation}
\end{prop}

\begin{proof} For partitions $\lambda$ and $\mu$,
take an $n$ which is sufficiently large 
(take $n \geq 2 \max(|\la|,|\mu|)$), then $(n-|\la|,\la)$
and $(n-|\mu|,\mu)$ are both partitions and this
scalar product can easily be computed on the irreducible character
basis by
\begin{equation}
\left<\phi_n(\st_\la), \phi_n(\st_\mu)\right> =
\left<s_{(n-|\la|,\la)}, s_{(n-|\mu|,\mu)}\right> = 
\delta_{\la\mu} = \left< \st_\la, \st_\mu \right>_@
~.
\end{equation}
Since this calculation holds on a basis,
Equation \eqref{eq:twoscalarrelations} holds for
all symmetric functions $f$ and $g$.
\end{proof}

This scalar product 
allows us to characterize the $\st_\lambda$ symmetric functions
as the orthonormal basis with respect to the scalar product $\left< \cdot, \cdot\right>_@$. This basis is triangular with respect the Schur basis ($\st_\lambda$ is equal to $s_\lambda$ plus terms of degree lower) and hence it may be calculated using
Gram-Schmidt orthonormalization with respect to this scalar product.
\end{subsection}
\end{section}

\begin{section}{Combinatorial rules for the products $h_\mu \st_\lambda$, $\st_{\mu_1} \st_{\mu_2} \cdots \st_{\mu_k} \st_\lambda$ and $\hht_{\mu_1} \hht_{\mu_2} \cdots \hht_{\mu_k} \st_\lambda$}
\label{sec:rules}

In this section we state the combinatorial rules for the
products, the proofs are in Section \ref{sec:proofs}.
The combinatorial interpretations we state here will follow from
a formula for $\hht_{\mu} \st_\lambda$ (Theorem 
\ref{thrm:htmustlam} and Corollary \ref{cor:htmpistla}) proved in Section \ref{sec:htstprod}. We state the following theorems in terms of increasing number of restrictions.

\begin{theorem} \label{thrm:hmustlam}
Let $\la$ and $\gamma$ be partitions and $\alpha$ a composition, then  the coefficient
of $\st_\ga$ in $h_{\alpha_1} h_{\al_2}
\cdots h_{\alpha_{\ell(\al)}} \st_\lambda$ is equal to the
number of $T \in \MCT_\gamma{(\la,\al)}$
such that $T$ is a lattice tableau.
\end{theorem}

We next state the combinatorial rule for the product 
of an irreducible character basis element and 
$\hht_{\mu_1} \hht_{\mu_2} \cdots \hht_{\mu_{\ell(\mu)}}$
%$\st_{\mu_1} \st_{\mu_2} \cdots \st_{\mu_k}$
and leave the proof for Section \ref{sec:proofhtprod}.
The difference between the previous combinatorial description
is that the labeled entries must be sets rather than multisets.

\begin{theorem}\label{thrm:htkmultiplestla}
Let $\lambda$ and $\gamma$ be partitions and $\alpha$ a composition, then the coefficient of $\st_\gamma$ in 
$\hht_{\al_1} \hht_{\al_2} \cdots \hht_{\al_{\ell(\al)}} \st_\lambda$
is equal to the number of $T \in \MCT_{\gamma}{(\la,\al)}$ such that
the entries of the tableaux are sets (no repeated entries) and $T$ is
a lattice tableau.
\end{theorem}

The expansion of the expression
$\st_{\mu_1} \st_{\mu_2} \cdots \st_{\mu_{\ell(\mu)}} \st_\lambda$
has a combinatorial interpretation where there is
only one extra condition that sets consisting of a single un-barred entry
may not appear in the first row (and we already have the condition that sets
consisting of a single barred entry may not occur in the first row
as part of the definition of $\MCT_{\gamma}{(\la,\al)}$).

\begin{theorem}\label{thrm:stkmultiplestla}
Let $\lambda$ and $\gamma$ be partitions and $\alpha$ a composition, then
the coefficient of $\st_\gamma$ in 
$\st_{\al_1} \st_{\al_2} \cdots \st_{\al_{\ell(\al)}} \st_\lambda$
is equal to the number of $T \in \MCT_{\gamma}{(\la,\al)}$ such that
the entries of the tableaux are sets (no repeated entries),
$T$ is a lattice tableau, and only
labels of sets of size greater than $1$ are allowed in the
first row.
\end{theorem}

\begin{example} \label{ex:allexample} The coefficient of
$\st_4$ in $h_{21} \st_{22}$ is equal to $8$.
The same coefficient in $\hht_{2} \hht_{1} \st_{22}$ is equal to $7$
and the coefficient
of $\st_4$ in $\st_2 \st_1 \st_{22}$ is equal to $5$.  These expressions
are represented
by the following $8$ tableaux (equal to the coefficient of
$\st_4 \hbox{ in } h_{21} \st_{22}$) such that
the first $7$ represent the coefficient of $\st_4 \hbox{ in } \hht_2 \hht_1 \st_{22}$
and the first $5$ represent the coefficient of
$\st_4 \hbox{ in } \st_{2}\st_{1} \st_{22}$.

$$\squaresize=14pt
\young{\o1&\o11&\o21&\o22\cr&&&\cr}\hskip .2in
\young{\o1&\o1&\o21&2\cr&&&&\o21\cr}\hskip .2in
\young{\o1&\o1&1&\o21\cr&&&&\o22\cr}\hskip .2in
\young{\o1&\o1&1&\o22\cr&&&&\o21\cr}$$
$$\squaresize=16pt
\young{\o1&\o1&\o21&\o212\cr&&&\cr}\hskip .2in
\squaresize=14pt
\young{\o1&\o1&\o21&\o22\cr&&&&1\cr}\hskip .2in
\young{\o1&\o1&\o21&\o21\cr&&&&2\cr}\hskip .2in
\squaresize=16pt
\young{\o1&\o1&\o211&\o22\cr&&&\cr}
$$
\end{example}

We will build our proof of Theorem \ref{thrm:stkmultiplestla} from
Theorem \ref{thrm:htmustlam} which is an expression
for the product $\hht_\mu \st_\la$. We will first complete
the statements of the expressions and then follow with the proofs.
\end{section}

\begin{section}{A combinatorial rule for the product
of $\hht_\mu \st_\la$}\label{sec:htstprod}

In developing a combinatorial interpretation for $\hht_\mu \st_\la$ 
we will start by enumerating a sequences of tableaux with
certain restrictions and show that they are in bijection with
multiset tableaux.

%We think of $\st_\la$ representing a single multiset valued tableau of
%shape $(r, \la)/(\la_1)$ since
%there is one term $\st_\la$ in the expansion of $h_\nu$ for each multiset
%tableau $T$ of shape $\la$ and content $\nu$.  We also think of $\hht_\mu$ as
%representing a
%multiset partition, $\pi$, such that $\mt(\pi) = \mu$.  The product $\hht_\mu \st_\la$
%when expanded in the irreducible character basis will have one term for every column strict tableau
%where a part of $\pi$ may be joined with an entry from the tableau $T$.

\begin{theorem} \label{thrm:htmustlam}
For partitions $\la$ and $\ga$ and a composition $\alpha$
let $\mu = sort(\alpha)$, $r = \ell(\la)$ and $\ell = \ell(\mu)$.
The coefficient
of $\st_\ga$ in $\hht_\mu \st_\la$ is equal to the
number of $T \in \MCT_\ga{(\la,\alpha)}$
whose entries are sets with
at most one non-barred entry (and at most one barred entry)
and $T$ is a lattice tableau.
\end{theorem}

Notice that the difference between the combinatorial interpretation
in Theorem \ref{thrm:htmustlam} and the one that appears in
Theorem \ref{thrm:htkmultiplestla}
is that here we require that there is at most one
barred entry and at most one unbarred entry in each set.

\begin{example} \label{ex:ht21st22} The coefficient of
$\st_4$ in $\hht_{21} \st_{22}$ is equal to $6$ since the tableaux described by the theorem are
\squaresize=12pt
$$\young{\o1&\o11&\o21&\o22\cr&&&\cr}\hskip .2in
\young{\o1&\o1&\o21&\o21\cr&&&&2\cr}\hskip .15in
\young{\o1&\o1&\o21&2\cr&&&&\o21\cr}\hskip .15in
\young{\o1&\o1&1&\o21\cr&&&&\o22\cr}\hskip .15in
\young{\o1&\o1&1&\o22\cr&&&&\o21\cr}\hskip .15in
\young{\o1&\o1&\o21&\o22\cr&&&&1\cr}
$$
Their respective reading words are 
$\o1.\o2\o1.\o2$, $\o1\o1.\o2\o2.$, $\o1\o1.\o2\o2.$,
$\o1\o1.\o2.\o2$, $\o1\o1.\o2.\o2$, $\o1\o1.\o2.\o2$, 
where the periods indicate the transitions of the reading words.
\end{example}

The coefficient of $\st_\ga$ in $\hht_\mu \st_\la$ is equal to the
scalar product $\left< \hht_\mu \st_\la, \st_\ga \right>_@$ and we 
apply Example 23 of Section I.7 of \cite{Mac}
to compute it
\begin{equation}\label{eq:stcoeffinhtstprod}
\left< \hht_\mu \st_\la,
\st_\ga \right>_@=
\sum_{\tau^{(0)} \vdash n-|\mu|} \sum_{\tau^{(1)} \vdash \mu_1}
\cdots \sum_{\tau^{(\ell)} \vdash \mu_\ell}
c_{\tau^{(0)} \tau^{(1)} \cdots \tau^{(\ell)}}^{(n-|\ga|,\ga)}
c_{\tau^{(0)} \tau^{(1)} \cdots \tau^{(\ell)}}^{(n-|\la|,\la)}~.
\end{equation}

\begin{definition}\label{def:Bset}
Assume that $\la$, $\ga$,
$\tau^{(i)}$ for $0 \leq i \leq \ell$ are partitions
such that $|\la| = |\ga| = \sum\limits_{i=0}^\ell |\tau^{(i)}|$.
Define $B_{\tau^{(0)}, \tau^{(1)},\ldots,\tau^{(\ell)}}^{\ga,\la}$ to be
the set of sequences of tableaux of the form
\begin{equation}\label{eq:seqform}
(T_0, T_1, T_2, \ldots, T_\ell)
\end{equation}
satisfying
\begin{itemize}
\item $T_0$ is the super-standard tableau of shape $\tau^{(0)}$. 
\item $T_i$ is a skew-shape column-strict tableau
with the outer shape of $T_{i-1}$ as the inner shape of $T_i$.
\item $T_i$ is jeu de Taquin equivalent to a tableau of shape $\tau^{(i)}$.
\item The word $read(T_0) read(T_1) \cdots read(T_\ell)$ has content $\dcl 1^{\la_1}, 2^{\la_2}, \ldots, \ell^{\la_{\ell(\la)}}\dcr$ and is lattice.
%and of content $\dcl 1^{\la_1}, 2^{\la_2}, \ldots, \ell^{\la_{\ell(\la)}}\dcr$.
\item $\ga$ is the outer shape of $T_\ell$.
\end{itemize}
\end{definition}

\begin{example} \label{ex:Btabseq}
Consider the following example with $\ga = (5,4)$, $\la = (5,2,2)$
and $|\tau^{(0)}| = 6$, $|\tau^{(1)}| = 2$
and $|\tau^{(2)}|=1$.

For the partitions $(\tau^{(0)}, \tau^{(1)}, \tau^{(2)}) = ((5,1), (2), (1))$, we have that
$\left|B_{(5,1)(2)(1)}^{(5,4)(5,2,2)}\right| = 1$ and contains:

$$\left(\raisebox{-5pt}{\young{2\cr1&1&1&1&1\cr}}\ ,\raisebox{-5pt}{\young{&2&3\cr&&&&\cr}}\ ,\raisebox{-5pt}{\young{&&&3\cr&&&&\cr}}\right)$$

For $(\tau^{(0)}, \tau^{(1)}, \tau^{(2)}) = ((4,2), (2), (1))$, we have that
$\left|B_{(4,2)(2)(1)}^{(5,4)(5,2,2)}\right| = 4$ and contains:

$$\left(\raisebox{-5pt}{\young{2&2\cr1&1&1&1\cr}}\ ,
\raisebox{-5pt}{\young{&&3&3\cr&&&\cr}}\ ,
\raisebox{-5pt}{\young{&&&\cr&&&&1\cr}}\right),
\left(\raisebox{-5pt}{\young{2&2\cr1&1&1&1\cr}}\ ,
\raisebox{-5pt}{\young{&&3\cr&&&&3\cr}}\ ,
\raisebox{-5pt}{\young{&&&1\cr&&&&\cr}}\right),$$
$$\left(\raisebox{-5pt}{\young{2&2\cr1&1&1&1\cr}}\ ,
\raisebox{-5pt}{\young{&&1&3\cr&&&\cr}}\ ,
\raisebox{-5pt}{\young{&&&\cr&&&&3\cr}}\right),
\left(\raisebox{-5pt}{\young{2&2\cr1&1&1&1\cr}}\ ,
\raisebox{-5pt}{\young{&&1\cr&&&&3\cr}}\ ,
\raisebox{-5pt}{\young{&&&3\cr&&&&\cr}}\right)$$

For $(\tau^{(0)}, \tau^{(1)}, \tau^{(2)}) = ((4,2), (1,1), (1))$, we have that $\left|B_{(4,2)(1,1)(1)}^{(5,4)(5,2,2)}\right| = 1$ and contains:

$$\left(\raisebox{-5pt}{\young{2&2\cr1&1&1&1\cr}}\ ,
\raisebox{-5pt}{\young{&&3\cr&&&&1\cr}}\ ,
\raisebox{-5pt}{\young{&&&3\cr&&&&\cr}}\right)$$
\end{example}

%$$\young{1&1&23&23\cr&&&&4\cr}\hskip .2in
%\young{1&1&23&4\cr&&&&23\cr}\hskip .2in
%\young{1&1&23&24\cr&&&&3\cr}\hskip .2in
%\young{1&1&3&23\cr&&&&24\cr}\hskip .2in
%\young{1&1&3&24\cr&&&&23\cr}\hskip .2in
%\young{1&13&23&24\cr&&&&\cr}$$

We will use the two versions of the Littlewood-Richardson rule
in Propositions \ref{prop:LR1} and \ref{prop:LR2} to show the following result.

\begin{prop} \label{prop:baselineCI}
Let $\la$, $\ga$, and $\tau^{(i)}$, $0 \leq i \leq \ell$,  be partitions satisfying 
$|\la| = |\ga| = \sum\limits_{i=0}^\ell |\tau^{(i)}|$, then 
$$\left|B_{\tau^{(0)} \tau^{(1)}\ldots\tau^{(\ell)}}^{\ga,\la}\right| =
c^\ga_{\tau^{(0)} \tau^{(1)}\cdots\tau^{(\ell)}} c^\la_{\tau^{(0)}\tau^{(1)}\cdots\tau^{(\ell)}}~.$$
\end{prop}

\begin{proof}
The enumeration will follow by induction on $\ell$.
First, we note that
$$c_{\tau^{(0)} \tau^{(1)} \cdots \tau^{(\ell)}}^{\ga}
= \sum_{\tau^{(0)} \subseteq \ga^{(1)} \subseteq \cdots \subseteq \ga^{(\ell-1)} \subseteq \ga} c_{\tau^{(0)}\tau^{(1)}}^{\ga^{(1)}} 
c_{\ga^{(1)}\tau^{(2)}}^{\ga^{(2)}}\cdots 
c_{\ga^{(\ell-1)}\tau^{(\ell)}}^{\ga^{(\ell)}}$$
where the sum is over chains of partitions $\ga^{(i)}$ for $1 \leq i \leq \ell-1$
such that $|\tau^{(i)}| = |\ga^{(i)}| - |\ga^{(i-1)}|$ with 
the convention that $\ga^{(0)} = \tau^{(0)}$
and $\ga^{(\ell)} = \ga$.
This follows by the associativity of products of Schur functions since both
sides of this equation are the coefficient of $s_{\ga}$ in 
$s_{\tau^{(0)}} s_{\tau^{(1)}} \cdots s_{\tau^{(\ell)}}$.  For the induction argument
we only need that
$$c_{\tau^{(0)} \tau^{(1)} \cdots \tau^{(\ell)}}^{\ga}
= \sum_{\zeta \subseteq \ga} 
c_{\tau^{(0)} \tau^{(1)} \cdots \tau^{(\ell-1)}}^{\zeta} 
c_{\zeta\tau^{(\ell)}}^{\ga}~.$$

Now for $\ell=0$ we are taking as a base case
the definition that $B_{\tau}^{\ga,\la}$ is empty unless
$\la = \ga = \tau$ and this case, $B_\tau^{\tau\tau} = \{(T)\}$
where $T$ is the superstandard tableau of shape $\tau$.  The base
case of the coefficient is the coefficient $c_\tau^\tau=1$.

%Consider the case when $\ell=1$.  The number of pairs $(T_0, T_1)$ such that
%the content of $T_0$ and $T_1$ together is $\la$ and the shape of $T_1$ 
%when brought to straight shape is $\tau^{(1)}$ and
%$read(T_0) read(T_1)$ is lattice is equal to $c_{\tau^{(0)}\tau^{(1)}}^\la$
%by Proposition \ref{prop:LR1}. Let $S_1$ a fixed straight shape tableau, 
%the number of $T_1$ which is skew shape
%$\ga/\tau^{(0)}$ and that is jeu de Taquin equivalent to $S_1$
%is equal to $c^\ga_{\tau^{(0)}\tau^{(1)}}$ by Proposition \ref{prop:LR2}.
%So the number of sequences $(T_0, T_1)$ in
%$B_{\tau^{(0)}\tau^{(1)}}^{\ga,\la}$ is equal to 
%$c_{\tau^{(0)}\tau^{(1)}}^\ga c_{\tau^{(0)}\tau^{(1)}}^\la$.
%%%%%%%

Now assume that the proposition is true for tableaux sequences of length $\ell-1$.  We observe that the number of sequences in $B_{\tau^{(0)} \tau^{(1)}\ldots\tau^{(\ell)}}^{\ga,\la}$ (see Definition \ref{def:Bset}) is the same as the number of pairs 
%and we consider how many sequences there are of the form of equation \eqref{eq:seqform1}
%that are in $A_{\tau^{(0)} \tau^{(1)}\ldots\tau^{(\ell)}}^{\ga,\la}$.  Clearly there are the same as the number of pairs
$$((T_1, T_2, \ldots, T_{\ell-1}), T_{\ell})$$
where $(T_1, T_2, \ldots, T_{\ell-1}) \in B_{\tau^{(0)} \tau^{(1)}\ldots\tau^{(\ell-1)}}^{\zeta,\nu}$
for some partitions $\nu \subseteq \la$ and $\zeta \subseteq \ga$ and the last
tableau $T_{\ell}$ has the following properties 
\begin{enumerate}
\item \label{B:one} $\tau^{(\ell)}$ is the straight shape of the tableau obtained by applying jeu de Taquin to $T_{\ell}$. 
%when it is brought to straight shape by jeu de Taquin
\item \label{B:two} $1^{\nu_1}2^{\nu_2}\cdots\ell(\nu)^{\nu_{\ell(\nu)}} 
read(T_{\ell})$ is a lattice word of content $\la$.
\item \label{B:three} $T_{\ell}$ is of shape $\ga/\zeta$.
\end{enumerate}

Let $S_\ell$ be the straight shape tableau obtained by applying jeu de Taquin to $T_\ell$, then $1^{\nu_1}2^{\nu_2}\cdots\ell(\nu)^{\nu_{\ell(\nu)}} 
read(T_{\ell})$ is lattice if and only if $1^{\nu_1}2^{\nu_2}\cdots\ell(\nu)^{\nu_{\ell(\nu)}} 
read(S_{\ell})$ is lattice.  The number of $T_\ell$ satisfying
\eqref{B:one}, \eqref{B:two} and \eqref{B:three}
is equal to the number of pairs $(T_\ell, S_\ell)$
satisfying:
\begin{enumerate}
\item[($1'$)] \label{Bp:one} $S_\ell$ is equal to the straight shape tableau obtained by applying jeu de Taquin to $T_\ell$. 
%$T_{\ell}$ is equal to $S_\ell$ when it is brought to straight shape by jeu de Taquin
\item[($2'$)] \label{Bp:two} $S_\ell$ is of shape $\tau^{(\ell)}$.
\item[($3'$)] \label{Bp:three} $1^{\nu_1}2^{\nu_2}\cdots\ell(\nu)^{\nu_{\ell(\nu)}}. 
read(S_{\ell})$ is a lattice word of content $\la$.
\item[($4'$)] \label{Bp:four} $T_{\ell}$ is of shape $\ga/\zeta$.
\end{enumerate}
By Proposition \ref{prop:LR1}
there are $c_{\tau^{(\ell)}\nu}^\lambda$ possible $S_\ell$ satisfying $(2')$ and $(3')$.
If we fix a straight shape $S_\ell$, then Proposition \ref{prop:LR2}
says that the number of $T_\ell$ satisfying $(1')$ and $(4')$
is equal to $c_{\tau^{(\ell)}\zeta}^{\ga}$.  Therefore we have
$c_{\tau^{(\ell)}\nu}^\la c_{\tau^{(\ell)}\zeta}^{\ga}$ pairs $(T_\ell, S_\ell)$.

By the inductive hypothesis $\left|B_{\tau^{(0)} \tau^{(1)}\ldots\tau^{(\ell-1)}}^{\zeta,\nu}\right|
= c_{\tau^{(0)} \tau^{(1)}\ldots\tau^{(\ell-1)}}^{\zeta}
c_{\tau^{(0)} \tau^{(1)}\ldots\tau^{(\ell-1)}}^{\nu}$, then
by the addition and multiplication principles we have
\begin{align*}
\left| B_{\tau^{(0)} \tau^{(1)}\ldots\tau^{(\ell)}}^{\ga,\la} \right| &=
\sum_{\zeta\subseteq\ga}\sum_{\nu \subseteq \la} c_{\tau^{(\ell)}\zeta}^{\ga} c_{\tau^{(\ell)}\nu}^\la
c_{\tau^{(0)} \tau^{(1)}\ldots\tau^{(\ell-1)}}^{\zeta}
c_{\tau^{(0)} \tau^{(1)}\ldots\tau^{(\ell-1)}}^{\nu}\\
&= c_{\tau^{(0)} \tau^{(1)}\ldots\tau^{(\ell)}}^{\ga}
c_{\tau^{(0)} \tau^{(1)}\ldots\tau^{(\ell)}}^{\la}~.
\end{align*}
The proposition follows by induction.
\end{proof}

We are now ready to prove the combinatorial interpretation of Theorem \ref{thrm:htmustlam} for $\hht_\mu\st_\la$.

\begin{proof} (of Theorem \ref{thrm:htmustlam}) Fix a positive integer $n$ 
which is sufficiently large.
The combinatorial interpretation for 
$B_{\tau^{(0)} \tau^{(1)} \cdots \tau^{(\ell)}}^{(n-|\ga|,\ga)(n-|\la|,\la)}$
in this proof will depend on $n$, but the cardinality of the set is
independent of $n$ as long as the number of $1$'s in $T_0$ is greater than $\ga_1$
for all sequences $T^\ast \in B_{\tau^{(0)} \tau^{(1)} \cdots \tau^{(\ell)}}^{(n-|\ga|,\ga)(n-|\la|,\la)}$.

In order to show that this combinatorial interpretation is
correct we will show that the number of tableaux in the statement of the
theorem equal to
Equation \eqref{eq:stcoeffinhtstprod}.
That is, we will show that the number of tableaux described in
Theorem \ref{thrm:htmustlam} is equal to
$$\left< \hht_\mu \st_\la,
\st_\ga \right>_@=\sum_{\tau^{(0)},\tau^{(1)},\ldots,\tau^{(\ell)}}
c_{\tau^{(0)} \tau^{(1)} \cdots \tau^{(\ell)}}^{(n-|\ga|,\ga)}
c_{\tau^{(0)} \tau^{(1)} \cdots \tau^{(\ell)}}^{(n-|\la|,\la)}
=
\sum_{\tau^{(0)},\tau^{(1)},\ldots,\tau^{(\ell)}}
\left|B_{\tau^{(0)} \tau^{(1)} \cdots \tau^{(\ell)}}^{(n-|\ga|,\ga)(n-|\la|,\la)}\right|$$
by Proposition \ref{prop:baselineCI}.
That is, we will show that there is a bijection between the set of tableaux described in the theorem and the set
\begin{equation}
B^{\ga,\la}_\alpha := 
\biguplus_{\tau^{(0)},\tau^{(1)},\ldots,\tau^{(\ell)}}
B_{\tau^{(0)} \tau^{(1)} \cdots \tau^{(\ell)}}^{(n-|\ga|,\ga)(n-|\la|,\la)}~,
\end{equation}
where the union is over all partitions $\tau^{(0)} \vdash n-|\mu|$
and $\tau^{(i)} \vdash \alpha_i$ and the
set $B_{\tau^{(0)} \tau^{(1)} \cdots \tau^{(\ell)}}^{(n-|\ga|,\ga)(n-|\la|,\la)}$
is defined in Definition \ref{def:Bset}.

As an example to follow along, we note 
that the tableaux in Example
\ref{ex:ht21st22} are in bijection
with the elements of $B^{(5,4),(5,2,2)}_{(2,1)}$ in Example \ref{ex:Btabseq}. The elements in both examples are given in order so that they correspond when we apply the bijection described below.  %in the same order that they are presented in the examples.
\vskip .2in

Let $T \in \MCT_\ga{(\la,\alpha)}$ have set entries with
at most one barred and at most one non-barred entry and such that
$T$ is a lattice tableau.
The tableau $T$ is
a skew tableau of shape $(r,\ga)/(\ga_1)$ and we think of the
cells in the inner shape $(\ga_1)$ as filled with empty sets. 
First pad $T$ with $n - |T|$ additional empty cells in the first row so that it is of
shape $(n-|\ga|,\ga)$ to make a tableau $T'$.  Next, let $T'|_-$ be the skew tableau consisting of the cells of $T'$ with no un-barred entries (including the empty cells) and $T'{\big|}_{i}$ for $1 \leq i \leq \ell$
be the skew tableau consisting of the cells of the form $\{i\}$ or $\{\o{j},i\}$. In $T'|_-$ replace empty cells with 1 and $\bar{j}$ with $j+1$ and in $T'{\big|}_{i}$ replace $\{i\}$ with a $1$ and $\{\o{j},i\}$ with a $j+1$. Then set 
$$T^\ast = (T'{\big|}_{-}, T'{\big|}_{1}, \ldots,
T'{\big|}_{\ell}),$$
that is $T_0 = T'{\big|}_{-}$ and $T_i = T'{\big|}_{i}$ for $1 \leq i \leq \ell$.

To ensure this correspondence is clear, we have for example (with $n=12$), (note that we padded $T_0$ with $n-|T|=2$ extra cells)
$$
\squaresize = 14pt
\raisebox{-12pt}{\young{\o2&\o32\cr\o1&\o11&\o22\cr&&&\o11&\o32\cr}} \
\longrightarrow
{\squaresize = 11pt
\ T_0 = \raisebox{-12pt}{\young{3\cr2\cr1&1&1&1&1\cr}\ ,}\  T_1 = \raisebox{-12pt}{\young{\cr&2\cr&&&&&2\cr}\ ,} \  T_2
= \raisebox{-12pt}{\young{&4\cr&&3\cr&&&&&&4\cr}}}~.
$$

To establish that $read(T_0) read(T_1) \cdots read(T_\ell)$
is a lattice word we note that if $w = read(T|_-)$ with $\o{j}$ replaced by $j+1$
then since no cells labeled with $\{\o{j}\}$ are allowed in the
first row of $T$, then $read(T_0) = 1^{n-|T|+\ga_1} w$. 
The fact that $read(T|_-) read(T|_1)\cdots read(T|_\ell)$
is a lattice word
implies that $read(T_0) read(T_1) \cdots read(T_\ell)$ is also lattice word
since the number of $2$'s in $read(T_0) read(T_1) \cdots read(T_\ell)$
will be less than the $n-|T|+\ga_1$ $1$'s at the beginning of the word.

The sequence $T^\ast$
is an element in $B^{\ga,\la}_\alpha$ and we can determine the sequence
of partitions $\tau^{(i)}$ from $T^\ast$ by setting
$\tau^{(i)}$ to be the shape of $T'{\big|}_{i}$ once it is
brought to straight shape using jeu de Taquin.
This gives us a sequence of partitions such that
$T^\ast \in B_{\tau^{(0)} \tau^{(1)} \cdots \tau^{(\ell)}}^{(n-|\ga|,\ga)(n-|\la|,\la)}$
by Definition \ref{def:Bset}.

Moreover, every $(T_0, T_1, \ldots, T_\ell) \in B^{\ga,\la}_\alpha$
corresponds to a set tableau $T$ of shape
$(n-|\ga|,\ga)$ by reversing the bijection
and replacing $1$ with a blank cell in $T_0$ and a $j+1$ with
a $\{\o{j}\}$, and in $T_i$ replacing a $1$ with a cell labeled
with an $\{i\}$ and a $j+1$ with $\{\o{j},i\}$
and overlaying the resulting
skew tableaux of shape $(n-|\ga|,\ga)$.
The result may have too many blank cells, but they
can be deleted until there are precisely
$\gamma_1$.

Notice that $read(T|_-) read(T|_1)\cdots read(T|_\ell)$ is equal to the word formed by starting with
$read(T_0) read(T_1) \cdots read(T_\ell)$ and then deleting the $1$'s and replacing
$j+1$ with $\o{j}$.  Thus, if $read(T_0) read(T_1) \cdots read(T_\ell)$ is a lattice
word then
$read(T|_-) read(T|_1)\cdots read(T|_\ell)$ will also be a lattice word.
\end{proof}

%\begin{example}
%The coefficient of
%$\st_{3}$ in $\hht_{11} \st_{22}$ is equal to $2$, then this coefficient
%is represented
%by the the following tableaux.
%\squaresize=14pt
%$$\young{\o1&\o1&\o21\cr&&&\o22\cr}\hskip .2in
%\young{\o1&\o1&\o22\cr&&&\o21\cr}$$
%\end{example}

%\begin{example}
%The coefficient of
%$\st_{4}$ in $\hht_{11} \st_{3}$ is equal to $4$, then this coefficient
%is represented
%by the the following tableaux.
%\squaresize=14pt
%$$\young{\o1&\o1&\o11&2\cr&&&\cr}\hskip .2in
%\young{\o1&\o1&1&\o12\cr&&&\cr}\hskip .2in
%\young{\o1&\o1&\o1&1\cr&&&&2\cr}\hskip .2in
%\young{\o1&\o1&\o1&2\cr&&&&1\cr}$$
%\end{example}

%\begin{example}
%The coefficient of
%$\st_{1}$ in $\hht_{11} \st_{11}$ is equal to $4$, then this coefficient
%is represented
%by the the following tableaux.
%\squaresize=14pt
%$$\young{\o11\cr&\o22\cr}\hskip .2in
%\young{\o22\cr&\o11\cr}\hskip .2in
%\young{\o1\cr&\o21&2\cr}\hskip .2in
%\young{\o1\cr&1&\o22\cr}$$
%\end{example}

To prove the other combinatorial interpretations presented
in this paper, we need Corollary \ref{cor:htmpistla}, (stated
below) where we establish a combinatorial interpretation for
the product $\hht_{\mt(\pi)} \st_\la$ where $\pi$ is a
multiset partition.

Fix $\pi$, a multiset partition of 
$\dcl 1^{\mu_1}, 2^{\mu_2}, \cdots, \ell^{\mu_\ell}\dcr$, and
let $S_1 < S_2 < \ldots < S_d$
be the (distinct) parts of the multiset
partition ordered in reverse lex
order and let $\alpha(\pi) = (\al_1, \al_2,\ldots,\al_d)$
be the composition of $\ell(\pi)$
of length $d$ such that $\alpha_i$ is the multiplicity of
$S_i$ in $\pi$.  The symbol
$\alpha(\pi)$ is defined so that $sort(\alpha(\pi)) = \mt(\pi)$.

For partitions $\la$ and $\ga$, we define $\MCT'_\ga(\la,\pi)$ to be the set of tableaux $T \in \MCT_\ga(\la, \mu)$ such that
the labels of the cells satisfy the following two conditions
\begin{itemize}
\item the multiset of multisets of unbarred entries (ignoring the barred entries and those that don't have any unbarred entries)
is equal to the multiset partition $\pi$;
\item $T$ is a lattice tableau.
\end{itemize}

\begin{example}
In the following tableau, if we ignore barred entries there are three boxes with unbarred entries.
The distinct multisets of unbarred entries are $S_1 =\dcl 1 \dcr$ and $S_2 =\dcl 1,2 \dcr$ (occurs twice).
\squaresize=16pt
$$\young{\o1&\o1&\o21&\o212\cr&&&&12\cr}$$
Since $read(T|_{-}) read(T|_{S_1}) read(T|_{S_2}) = \o1\o1.\o2.\o2$ is lattice, then this tableau is an element
of $\MCT'_{(4)}((2,2),\dcl\dcl1\dcr,\dcl1,2\dcr,\dcl1,2\dcr\dcr)$.
%%%% moved to the definition of 'lattice tableau'
%Note that (for instance) the tableau
%\squaresize=14pt
%$$\young{\o1&\o21&\o21&\o12\cr&&&\cr}$$
%is column strict and an element of $\MCT_{(4)}((2,2),(2,1))$
%but has reading word of the barred entries equal to
%$read(T|_-)read(T|_1)read(T|_2)= \o1.\o2\o2.\o1$
%which is not lattice and hence it is not an element of
%$\MCT'_{(4)}((2,2),\dcl\dcl1\dcr,\dcl1\dcr,\dcl2\dcr\dcr)$.
\end{example}

\begin{cor} \label{cor:htmpistla} For $\la$ a partition and
$\pi \mvdash \dcl 1^{\mu_1}, 2^{\mu_2}, \cdots, \ell^{\mu_\ell}\dcr$,
\begin{equation}
\hht_{\mt(\pi)} \st_\la = \sum_{\ga}
\sum_{T \in \MCT'_\ga(\la, \pi)} \st_{\ga}
\end{equation}
where the sum is over all partitions $\ga$ such that  $|\ga|\leq \ell(\pi) + |\la|$.
\end{cor}

\begin{proof} By Theorem \ref{thrm:htmustlam}, to  demonstrate this result, we need to establish a bijection for each partition $\ga$ between
the set of $T \in \MCT_\ga(\la,\alpha(\pi))$
and $T' \in \MCT'_\ga(\la, \pi)$ such that
$$read(T'|_-) read(T'|_{S_1}) \cdots read(T'|_{S_d})=
read(T|_-)read(T|_1)\cdots read(T|_d)~.$$

We take the $S_1 < S_2 < \ldots < S_d$
be the (distinct) parts of the multiset
partition ordered in reverse lex order. The bijection between these sets is to replace the unbarred label $i$ in
$T$ with the set $S_i$ for each $1\leq i \leq d$
to obtain $T'$.  The inverse bijection is to replace the occurrences of the set $S_i$ in $T'$ with the label $i$ to obtain the corresponding 
$T \in \MCT_{\overline{sh(T')}}(\la,\al(\pi))$.

Hence we have by Theorem \ref{thrm:htmustlam} that
\begin{equation}
\hht_{\mt(\pi)} \st_\la = \sum_{\ga : |\ga|\leq|\mt(\pi)| + |\la|}
\sum_{T \in \MCT_\ga(\la, \alpha(\pi))} \st_\ga =
\sum_{\ga : |\ga|\leq|\mt(\pi)| + |\la|}
\sum_{T \in \MCT'_\ga(\la, \pi)} \st_\ga~.
\end{equation}
\end{proof}

\begin{example} We provide an example to clarify the notation in the bijection of the proof of Corollary \ref{cor:htmpistla}.
Let $\mu = (5,1)$ and then $\pi = \dcl \dcl1\dcr,
\dcl1,1\dcr, \dcl1,1\dcr, \dcl2\dcr\dcr$ is a multiset
partition of $\dcl1^5,2\dcr$ and
\begin{equation*}
\squaresize=14pt
T = \raisebox{-16pt}{\young{\o2&\o12&\o13\cr
\o1&\o1&1\cr&&&\o22\cr}}
\end{equation*}
is an element of $\MCT_{(3,3)}{((4,2),(1,2,1))}$.
Since the multisets of $\pi$ are ordered
$\dcl1\dcr<\dcl1,1\dcr<\dcl2\dcr$, the
corresponding $T'$ is equal to
\begin{equation*}
\squaresize=16pt
T' = \raisebox{-17pt}{\young{\o2&\o111&\o12\cr
\o1&\o1&1\cr&&&\o211\cr}}
\end{equation*}
and this is an element of $\MCT'_{(3,3)}((4,2), \pi) \subseteq 
\MCT_{(3,3)}{((4,2),(5,1))}$.  We also have
$read(T|_-) read(T|_1) read(T|_2) read(T|_3) =
read(T'|_-) read(T'|_{S_1})
read(T'|_{S_2}) read(T'|_{S_3})
= \o1\o1\o2..\o2\o1.\o1~,$ where $S_1 = \dcl 1\dcr$, $S_2=\dcl 1,1\dcr$ and $S_3=\dcl 2\dcr$. 
\end{example}
\end{section}

\begin{section}{Proofs of Theorems \ref{thrm:hmustlam}, \ref{thrm:htkmultiplestla} and \ref{thrm:stkmultiplestla}}
\label{sec:proofs}
\begin{subsection}{A proof of Theorem \ref{thrm:hmustlam}}
Recall that Theorem \ref{thrm:hmustlam} gives a combinatorial expression for the coefficients in the expansion of $h_\mu \st_\la$ in the irreducible
character basis, $\st_\la$.
We will rely on Corollary \ref{cor:htmpistla}
and Equation \eqref{eq:htdef}.

\begin{proof} (Theorem \ref{thrm:hmustlam})
The symmetric group $S_\ell$ acts on the multiset
$\dcl 1^{\al_1}, 2^{\al_2}, \cdots, \ell^{\al_\ell}\dcr$
as well as multiset partitions 
$\pi \mvdash \dcl 1^{\al_1}, 2^{\al_2}, \cdots, \ell^{\al_\ell}\dcr$
by permuting the values of the multiset and multiset partitions.
It is not hard to see
that $\mt(\sigma(\pi)) = \mt(\pi)$ for all $\sigma \in S_\ell$.
This implies that the obvious analogue of Equation \eqref{eq:htdef}
holds for compositions $\alpha$, that is,
$$h_{\alpha_1} h_{\alpha_2} \cdots h_{\alpha_\ell} = 
\sum_{\pi \mvdash \dcl 1^{\al_1} 2^{\al_2} \cdots \ell^{\al_\ell}\dcr}
\hht_{\mt(\pi)}~.$$

Notice that by Corollary \ref{cor:htmpistla}.
We have,
\begin{align*}
h_{\alpha_1} h_{\alpha_2} \cdots h_{\alpha_\ell} \st_\la 
&= \sum_{\pi \mvdash \dcl 1^{\al_1} 2^{\al_2} \cdots \ell^{\al_\ell}\dcr}
\hht_{\mt(\pi)} \st_\la\\
&=\sum_{\pi \mvdash \dcl 1^{\al_1} 2^{\al_2} \cdots \ell^{\al_\ell}\dcr}
\sum_\ga
\sum_{T \in \MCT'_\ga(\la,\pi)}
\st_{\ga}
\end{align*}
where the sum is over all partitions $\ga$ such that $|\ga| \leq |\la| + \ell(\pi)$.

This establishes exactly the conditions of the statement of Theorem \ref{thrm:hmustlam}
where the labels of the tableaux have labels whose entries have at most one barred entry and a multiset of non-barred entries. And if we read the barred entries using the reverse lex order
on the distinct multisets consisting of the non-barred entries the reading word is lattice.
\end{proof}
\end{subsection}

\begin{subsection}{A proof of Theorem \ref{thrm:htkmultiplestla}}\label{sec:proofhtprod}

We next establish a combinatorial interpretation for the terms
in $\hht_{\mu_1} \hht_{\mu_2} \cdots \hht_{\mu_k} \st_\lambda$.

\begin{proof} (Theorem \ref{thrm:htkmultiplestla})
Proposition 21 of \cite{OZ16} allows us to expand the product $\hht_\la\hht_\mu$ in terms of the induced character basis elements, i.e.,  $\hht$-basis.  In particular, a product of the form
$\hht_{\al_1} \hht_{\al_2} \cdots \hht_{\al_k}$ is equal to
the sum over all set partitions of a multiset $\dcl 1^{\al_1}, 2^{\al_2}, \ldots, k^{\al_k}\dcr$
(that is, multiset partitions where the sets in the partition do not
contain repeated elements), hence we have the expansion
\begin{equation} \label{eq:htproduct}
\hht_{\al_1} \hht_{\al_2} \cdots \hht_{\al_k} = \sum_{\pi} \hht_{\mt(\pi)}
\end{equation}
where the sum is over all set partitions $\pi$ of the 
multiset $\dcl 1^{\al_1}, 2^{\al_2}, \ldots, k^{\al_k}\dcr$ and $\mt(\pi)$ is the partition of the number $\ell(\pi)$ consisting of multiplicities of each multiset occurring in $\pi$.

It follows from Corollary \ref{cor:htmpistla} that
\begin{equation}
\hht_{\al_1} \hht_{\al_2} \cdots \hht_{\al_k} \st_\la =
\sum_\pi \sum_\ga \sum_{T \in \MCT'_\ga(\la, \pi)} \st_\ga
\end{equation}
where the first sum is over set partitions $\pi$ of $\dcl 1^{\al_1}, 2^{\al_2}, \ldots, k^{\al_k}\dcr$.

For a fixed $\ga$, there is one term in this sum for each set valued tableau in $\MCT_\ga(\la,\al)$
where the entries are sets and the condition that
if $S_1 < S_2 < \ldots < S_d$ are the sets
of the unbarred entries that appear
(ignoring the barred entries), then the word
$$read(T|_{-}) read(T|_{S_1}) read(T|_{S_2}) \cdots read(T|_{S_d})$$ 
is lattice because $T$ is in $\MCT'_\ga(\la,\pi)$.
\end{proof}

\end{subsection}

\begin{subsection}{A proof of Theorem \ref{thrm:stkmultiplestla}}\label{sec:proofstprod}

Building on Theorem \ref{thrm:htkmultiplestla} we can
develop a combinatorial expression for the coefficients
in the expansion of $\st_{\al_1} \st_{\al_2} \cdots \st_{\al_k} \st_\la$ in the irreducible
character basis.
We note that
$\hht_k = \st_k + \st_{k-1} + \cdots + \st_{()}$ by Equation
\eqref{eq:stdef} and so by induction we have that
$\st_{()} = \hht_{()} = 1$, $\st_1 = \hht_1 - \hht_{()}$
and
$\st_k = \hht_{k} - \hht_{k-1}$ for $k > 1$.

\begin{proof} (Theorem \ref{thrm:stkmultiplestla})
We replace $\st_k$ by $\hht_k - \hht_{k-1}$ for each term $\st_{\al_i}$ and expand.
\begin{align}\st_{\al_1} \st_{\al_2} \cdots \st_{\al_k} \st_\la&=
(\hht_{\al_1} - \hht_{\al_1-1})(\hht_{\al_2} - \hht_{\al_2-1})\cdots
(\hht_{\al_k} - \hht_{\al_k-1}) \st_\la\nonumber\\
&= \sum_{S \subseteq [k]} (-1)^{|S|}
\hht_{\al_1 - \chi(1 \in S)} \hht_{\al_2 - \chi(2 \in S)}\cdots 
\hht_{\al_k - \chi(k \in S)}\st_\la
\label{eq:htprodst}
\end{align}

By Equation \eqref{eq:htproduct}, we have that 
Equation \eqref{eq:htprodst} is equal to the following sum which we can expand
using Corollary \ref{cor:htmpistla},
\begin{equation}
= \sum_{S \subseteq [k]} \sum_\pi (-1)^{|S|} \hht_{\mt(\pi)} \st_\la
= \sum_{S \subseteq [k]} \sum_\pi \sum_\ga \sum_{T \in \MCT'_\ga(\la, \pi)} (-1)^{|S|}
\st_\ga \label{eq:fullsum}.
\end{equation}
where the second sum is over all set partitions $\pi$ of 
$\dcl 1^{\al_1 - \chi(1 \in S)},
2^{\al_2 - \chi(2 \in S)}, \ldots, k^{\al_k - \chi(k \in S)} \dcr$ and the third sum is over all partitions 
$\ga$ such that $|\ga| \leq |\la|+ \ell(\pi)$.

Now for each subset $S \subseteq [k]$, we can define an injective map
$$\phi_S : \biguplus_{\pi } \biguplus_{\ga} \MCT'_\ga(\la, \pi) \rightarrow \biguplus_{\pi'}
\biguplus_{\ga} \MCT'_\ga(\la, \pi')$$
where on the left we have the union over set partitions $\pi$ of the multiset
$$\dcl 1^{\al_1 - \chi(1 \in S)}, 2^{\al_2 - \chi(2 \in S)}, 
\ldots, k^{\al_k - \chi(k \in S)}\dcr$$
while on the right the union is over set partitions $\pi'$ of the
multiset $\dcl 1^{\al_1},2^{\al_2}, \ldots, k^{\al_k}\dcr$.
For an element $T \in \MCT'_\ga(\la, \pi)$
then $\phi_S(T)$ is equal to the tableau where a cell with a label of $\dcl i\dcr$ for each $i \in S$
is added to the first row.  Note that $\overline{sh(T)} = \overline{sh(\phi_S(T))} = \gamma$
and the reading words of the barred entries do not change.  Therefore, $\phi_S$ maps the
set $\MCT'_\ga(\la, \pi)$ to $\MCT'_\ga(\la, \pi \uplus \dcl\dcl i\dcr : i \in S\dcr)$.

We note that $\phi_S$ is an injective map and $T \in \MCT_\ga(\la,\al)$ is in the
image of $\phi_S$ if and only if there are cells with a label of $\dcl i\dcr$ where $i \in S$
in the first row.
Let $a(T)$ be the set of $i$ such that $\dcl i \dcr$ is a label in the first row of 
a $T \in \MCT_\ga(\la,\al)$.
By applying this injection, we can interchange the sum over subsets of $S \subseteq [k]$
in Equation \eqref{eq:fullsum} so that we first sum over all multiset tableaux and
for each fixed tableau, $T$, there is one term for each subset $S$ of the set
$a(T)$.  We write this as
\begin{equation}
= \sum_{\pi'} \sum_\ga \sum_{T \in \MCT'_\ga(\la, \pi')} \sum_{S \subseteq a(T)}  (-1)^{|S|} \st_\ga~.
\label{eq:switch}
\end{equation}

Since the summand $\st_\ga$ is independent of the sum over subsets
$S$, the terms with $a(T)$ not equal to the
empty set will sum to zero.  Therefore Equation \eqref{eq:switch} is equal to the expression
$$=\sum_{T} \st_{\overline{sh(T)}}$$
where this sum is over multiset tableaux $T \in \MCT_\ga(\la,\al)$
for some partition $\ga$ such that
\begin{itemize}
\item
the labels of un-barred entries (ignoring the barred entries) are sets, that is, they have no
repeated entries; 
\item $T$ is a lattice tableau;
\item the set $a(T)$ is empty, that is, there are no labels of size less than $2$ in
the first row.
\end{itemize}
These are equivalent to the conditions stated in Theorem \ref{thrm:stkmultiplestla}.
\end{proof}
\end{subsection}
\end{section}

\begin{section}{Concluding remarks and applications}
\label{sec:connections}

There is a formula for the stable Kronecker coefficients
in terms of Littlewood-Richardson coefficients and Kronecker
coefficients due to Littlewood \cite{Lit2}.
It says that
\begin{equation}
\overline{g}_{\alpha\beta}^\gamma=
\sum g_{\delta\epsilon\zeta} c_{\delta\sigma\tau}^{\alpha} c_{\epsilon\rho\tau}^{\beta} c_{\zeta\rho\sigma}^{\gamma}
\end{equation}
where the sum is over all partitions $\delta, \epsilon, \zeta,
\rho, \tau$ and $\sigma$.
This equation has been rediscovered and reproved in the 
literature (see for example \cite{BDO, Brion, ButlerKing, ScharfThibon, 
Thibon, ScharfThibonWybourne, Wagner}).

In the case when $\gamma$ is a single row or
column then this formula simplifies to an equation only involving
Littlewood-Richardson coefficients.
The combinatorial interpretation for the product $\st_k \st_\la$
in this paper reduces to Littlewood's formula in that case.

Our main goal in developing the results in this paper
is to provide a combinatorial answer to
two fundamental open problems in representation theory,
a formula for $\o{g}_{\la\mu\nu}$ and the multiplicity of
an irreducible symmetric group representation in an irreducible
polynomial $GL_n$ representation.  These correspond to the coefficients
of $\st_\nu$ in the expressions $\st_\mu \st_\la$ and $s_\mu$
(respectively).  The inequalities discussed in the introduction
may be extended to the following diagram.
\begin{center}
\begin{tabular}{ccccc}
$\st_{\mu_1} \st_{\mu_2} \cdots \st_{\mu_\ell} \st_\la$& $\leq$ & $\hht_{\mu_1} \hht_{\mu_2} \cdots \hht_{\mu_\ell} \st_\la$ & $\leq$ & $h_\mu \st_\la$\\
\begin{rotate}{-90}$\geq$\end{rotate}& &\begin{rotate}{-90}$\geq$\end{rotate}& &
\begin{rotate}{-90}$\geq$\end{rotate}\\\\
$\st_\mu \st_\la$ & $\leq$ & $\hht_\mu \st_\la$ & $\leq$ & $s_\mu \st_\la$
\end{tabular}
\end{center}
We therefore believe that if such a combinatorial formula exists
that it can be formulated in terms of the tableaux introduced in this paper.

\begin{subsection}{Multiset tableaux and paths in a Bratteli diagrams}
The coefficient of $\st_\lambda$ in the
symmetric function expression $(\hht_1)^r$ is equal to the
dimension of an irreducible representation
of the partition algebras \cite{HalRam,Martin3,Martin4} indexed by the partition $\lambda$.
Previously in the literature, these dimensions have mainly
been enumerated in terms
of paths in a Bratteli diagram and these are sometimes referred
to in the literature as vascillating tableaux
\cite{HalLew, MartinRollet}.  A consequence of our results is that it is much more natural to enumerate these in terms of
set valued tableaux since the combinatorial interpretation
more clearly reflects the generalization from standard tableaux
for dimensions of symmetric group modules to set valued
tableaux for dimensions of partition algebra modules.

Similarly, the dimensions of irreducible quasi-partition
modules \cite{DO} are equal to the coefficient of
$\st_\lambda$ in the symmetric function $(\st_1)^r$, these dimensions were first computed in \cite{ChauveGoupil} and the objects they enumerate correspond to paths in the Bratteli diagram of the quasi-partition algebra, these paths were called Kronecker tableaux.

In a recent paper, Benkart, Halverson and Harman \cite{BHH} gave
formulas for dimensions of irreducible partition algebra modules
in terms of Stirling numbers and Kostka numbers reflecting this connection with set valued tableaux.

\end{subsection}

\begin{subsection}{Bases of symmetric functions, set tableaux
and Grothendeick symmetric functions}
A consequence of the formulae that we present in this paper is that the
families of symmetric functions 
$\{\hht_{\mu_1} \hht_{\mu_2} \cdots \hht_{\mu_{\ell(\mu)}}\}_\mu$
and 
$\{\st_{\mu_1} \st_{\mu_2} \cdots \st_{\mu_{\ell(\mu)}}\}_\mu$
are two non-homogenous multiplicative bases of the ring of symmetric functions.
As special case of 
Theorem \ref{thrm:htkmultiplestla} and Theorem \ref{thrm:stkmultiplestla}
we have a combinatorial
interpretations for the transition coefficients between
these bases and the $\st$-basis.

There are a several ways of describing the $S_n$-module with character $\hht_{\mu_1} \hht_{\mu_2} \cdots \hht_{\mu_{\ell(\mu)}}$. It is a module of dimension

%There may be many ways of thinking of these symmetric functions  in terms  of symmetric group representation theory, but the symmetric function
%$\hht_{\mu_1} \hht_{\mu_2} \cdots \hht_{\mu_{\ell(\mu)}}$
%is the character of a module of dimension
$$\pchoose{n}{\mu_1} \pchoose{n}{\mu_2} \cdots \pchoose{n}{\mu_{\ell(\mu)}}$$
%\todo{do we want to be more specific  with the description of the module?
%MZ: Well there are lots of ways of describing this.
%%One description is the tensor of the
%%induced modules, but another is a subset module 
%%$v_{S_1} \otimes v_{S_2} \otimes \cdots \otimes v_{S_\ell}$ 
%%where $S_i \subseteq \{1,2, \ldots, n\}$
%%and $|S_i| = \mu_i$ and yet another is some sort of coset module
%The simplest might be the action on the $S_n$ module of
%the quotient group
%$S_{n} \times \cdots \times S_n/(S_{\mu_1}\times S_{n-\mu_1}) \times \cdots %\times (S_{\mu_\ell}\times S_{n-\mu_{\ell(\mu)}})$}
and
the symmetric function $\st_{\mu_1} \st_{\mu_2} \cdots \st_{\mu_{\ell(\mu)}}$
is the character of the symmetric group module
$$\SS^{(n-\mu_1,\mu_1)} \otimes \SS^{(n-\mu_2,\mu_2)}
\otimes \cdots \otimes \SS^{(n-\mu_{\ell(\mu)},\mu_{\ell(\mu)})}$$
where $\SS^\lambda$ is an irreducible symmetric group module indexed by
a partition $\lambda$ and the symmetric group acts diagonally on the tensors.

For a partition $\mu = (\mu_1, \ldots, \mu_{\ell(\mu)})$,
$$\hht_{\mu_1} \hht_{\mu_2} \cdots \hht_{\mu_{\ell(\mu)}} = \sum_\nu a_{\mu,\nu} \hht_\nu$$
where $a_{\mu,\nu}$ is equal to the number of set valued partitions
whose multiplicity of the sets is equal to $\nu$ and whose
content is equal to 
$\dcl 1^{\mu_1}, 2^{\mu_2}, \ldots, \ell^{\mu_{\ell}}\dcr$.

%The change of basis between
%$\hht_{\mu_1} \hht_{\mu_2} \cdots \hht_{\mu_{\ell(\mu)}}$
%and $\hht_\nu$ is equal to the number of set valued partitions
%whose multiplicity of the sets is equal to $\nu$ and whose
%content is equal to 
%$\dcl 1^{\mu_1}, 2^{\mu_2}, \ldots, \ell^{\mu_{\ell}}\dcr$.
And as a consequence of
Equation \eqref{eq:stdef}
$$\hht_{\mu_1} \hht_{\mu_2} \cdots \hht_{\mu_{\ell(\mu)}} = \sum_\nu b_{\mu,\nu} \st_\nu$$
where $b_{\mu,\nu}$ is equal to the number of set
valued tableaux of shape $(r,\nu)/(\nu_1)$ and
content $\dcl 1^{\mu_1}, 2^{\mu_2}, \ldots, \ell^{\mu_{\ell}}\dcr$.

%the
%coefficient of $\st_\nu$ in $\hht_{\mu_1} \hht_{\mu_2} \cdots \hht_{\mu_{\ell(\mu)}}$ is equal to the number of set
%valued tableaux of shape $(r,\nu)/(\nu_1)$ and
%content $\dcl 1^{\mu_1}, 2^{\mu_2}, \ldots, \ell^{\mu_{\ell}}\dcr$.

This family of symmetric functions has close connections (but they are not exactly the same)
to (skew) Grothendeick symmetric functions.  The coefficient
of a monomial symmetric function $m_\mu$ in a Grothendeick
symmetric function indexed by the partition
$\lambda$ or the skew tableaux $\lambda/\nu$) 
(see for intstance \cite{Buch} for a combinatorial 
formula for a monomial expansion
of symmetric functions $G_{\lambda/\mu}$ in terms of set valued tableaux)
is equal to $(-1)^{|\mu|-|\la|}$
times the number of set valued
tableaux of shape $\lambda$ and content $\mu$.

This connection demonstrates how these symmetric functions
are closely related
but the fact that the first row can be of different
sizes means that (for example) the tableaux 
\squaresize=14pt
$\raisebox{-5pt}{\young{12\cr}}$
and $\raisebox{-5pt}{\young{1&2\cr}}$ contribute a $2$ to the coefficient
of $\st_{()}$ in $\hht_{1} \hht_1$ while the same tableaux contribute
a term of $-m_{11}$ in
$G_{1}$ and a term of $+m_{11}$ in $G_{2}$.
\end{subsection}

\begin{subsection}{Quantum entanglement}
Multiple products of Kronecker coefficients have been considered in the literature with relation to quantum information theory.
A measure of quantum entanglement of qubits can be
reformulated as the calculation of
repeated Kronecker products of the symmetric functions
$s_{d,d}$ or $h_{d,d}$
\cite{GWXZ1, GWXZ2, GMWX, LT1, LT2, Wal1, Wal2}
and our combinatorial interpretation
can be useful in this application.

In particular, the approach taken in \cite{GMWX} to calculating the
measurement of qubits is to compute the multiplicities in the
Kronecker products of $h_{d,d}$ and $h_{d+1,d-1}$ and express
the $k$-fold Kronecker product
$\underbrace{s_{d,d} \ast s_{d,d} \ast \cdots \ast s_{d,d}}_{k\hbox{ times}}$
in terms of these expressions.

A consequence of Theorem \ref{thrm:htkmultiplestla} is that
the coefficient of $\st_{()}$ in the expression
$(\hht_d)^a (\hht_{d-1})^{k-a}$ is equal to the the number of set
valued tableau of shape $(r)$
and the content is equal to $\dcl 1^d, 2^d,
\ldots, a^d,(a+1)^{d-1},\ldots,k^{d-1}\dcr$.  If we take the
image of this symmetric function in the Frobenius map
$\phi_{2d}$ (from equation \eqref{eq:frob})
then there will be some cancellation of terms
but the resulting combinatorial interpretation for the coefficient
of $s_{2d}$ in this expression simplifies to the expression
stated in the following proposition.

\begin{prop}
The coefficient of the Schur function $s_{2d}$ in the
Kronecker product
$$\underbrace{h_{d,d} \ast h_{d,d} \ast \cdots \ast h_{d,d}}_{a\hbox{ times}}
\ast \underbrace{h_{d+1,d-1} \ast h_{d+1,d-1} \ast \cdots \ast h_{d+1,d-1}}_{k-a\hbox{ times}}$$
is equal to the number of set valued tableaux of shape $(r)$
where $r \leq 2d$ and the content is equal to $\dcl 1^d, 2^d,
\ldots, a^d,(a+1)^{d-1},\ldots,k^{d-1}\dcr$.
\end{prop}
\end{subsection}

\begin{subsection}{A relationship between the irreducible
character basis and plethysm} 
\label{restriction-plethysm}

% Another open problem relates to the multiplicities when we restrict a $GL_n$ representation to $S_n$. These multiplicities, $r_{\la\mu}$, are the change of basis coefficients between the Schur basis and the irreducible character basis, i.e, 
%$$s_\mu = \sum_{\la} r_{\la\mu} \st_\la$$
%and have appeared in the literature \cite{Lit2, King, ScharfThibon, Nis, ButlerKing}. %Although, they are known to have an expression in terms of the plethysm operation, see Section \ref{restriction-plethysm} for more details, it remains an open problem to find a combinatorial interpretation for them.

Littlewood \cite{Lit2} proved an expression for
the multiplicity of an irreducible $S_n$ module indexed by
the partition $(n-|\la|,\lambda)$ in an irreducible polynomial
$GL_n$ module indexed by the partition $\nu$.
Scharf and Thibon \cite{ScharfThibon}
proved this formula using modern notation and symmetric
function techniques to show explicitly that
it was equal to
$$\left< s_\nu, 
s_{(n - |\lambda|,\lambda)}[1+h_1+h_2+h_3+\cdots] \right>~.$$
Here we are using $f[g]$ to represent the plethysm of two
symmetric functions (see \cite{Mac, ScharfThibon}).

Since this multiplicity is also equal to the coefficient
of an irreducible character basis element $\st_\la$ in the
a Schur function $s_\lambda$, then we have (again,
for $n$ sufficiently large) that
$$s_\nu = \sum_\lambda
\left< s_\nu, 
s_{(n - |\lambda|,\lambda)}[1+h_1+h_2+h_3+\cdots] \right> \st_\la$$
and hence we can extend this symmetric function expression
linearly and we conclude that for any symmetric function $f$,
\begin{equation}
f = \sum_\lambda
\left< f, 
s_{(n - |\lambda|,\lambda)}[1+h_1+h_2+h_3+\cdots] \right> \st_\la~.
\end{equation}
We have a similar formula using the scalar product, $\left<, \right>_@$,
introduced in Section \ref{sec:scalarproduct}. Since the $\st$-basis is self 
dual with respect to this scalar product (see equation \eqref{eq:scalarproduct}),
\begin{equation}
f = \sum_\lambda
\left< f, 
\st_\la \right>_@ \st_\la~.
\end{equation}
We can conclude by linearity that
$$\left< f, g \right>_@ = 
\left< f, \phi_n(g)[1+h_1+h_2+h_3+\cdots] \right>$$
where $\phi_n$ is the Frobenius map defined in equation \eqref{eq:frob}.

\end{subsection}

\begin{subsection}{Characterizations by Pieri rules}
One reason to focus on the Pieri rule of a symmetric
function is that it provides a way of defining or
characterizing a basis.

\begin{definition} The family of symmetric functions
$\{\st_\la\}_\lambda$ may be defined
recursively as the unique set of symmetric functions
satisfying $\st_{()} = 1$ and
$$\st_\lambda = h_{\la_1} \st_{\overline{\lambda}} - \sum_T \st_{\overline{sh(T)}}$$
where the sum is over all $T \in \MCT(\overline{\la}, (k))$
with $\overline{sh(T)} \neq \la$ and
such that $T$ is a lattice tableau.
\end{definition}

\begin{example} We use this definition to calculate the expansion of
$\st_\la$ in the complete homogeneous basis for $|\la| \leq 3$.
For example,
\begin{itemize}
\item $\st_1 = h_1 - \st_{()} = h_1 - 1$.
\item $\st_2 = h_2 - (2\st_1 + 2\st_{()})
 = h_2 - 2 h_1$.
\item
$\st_{11} = h_1 \st_1 -(\st_2 + 2\st_1 + \st_{()})
= h_{11} - h_2 -h_1 +1$.
\item $\st_3 = h_3 - (\st_{11} + 2 \st_2 + 4\st_1 + 3 \st_{()})
= h_3 - h_2 - h_{11} + h_1$.
\item
$\st_{21} = h_2 \st_1 - (\st_3 + 3 \st_2 
+ 3 \st_{11} + 5 \st_1 + 2 \st_{()})= h_{21} - h_3 - 2 h_{11} + 3h_1$.
\item $\st_{111} = h_1 \st_{11} - (\st_{21} + \st_{2} + 2\st_{11} + \st_1)
= h_{111} - 2h_{21} + h_3 + h_2 - h_{11} + h_1 - h_{()}$.
\end{itemize}
\end{example}

We can similarly define $\{ \st_\la\}_\la$ in through the Pieri
rules that are given in Theorem
\ref{thrm:htkmultiplestla}
and Theorem \ref{thrm:stkmultiplestla}.
\end{subsection}

\end{section}

\end{document}